\documentclass[twoside,a4paper,12pt,centertags]{amsart}
\usepackage{amsmath,amssymb,verbatim,vmargin}
\usepackage[bookmarks=true]{hyperref}   

\theoremstyle{plain}
\newtheorem{thm}{Theorem}[section]
\newtheorem{lem}[thm]{Lemma}
\newtheorem{cor}[thm]{Corollary}
\newtheorem{prop}[thm]{Proposition}

\theoremstyle{definition}
\newtheorem{defn}[thm]{Definition}
\newtheorem{rem}[thm]{Remark}

\title[Solvability of elliptic systems]
{Solvability of elliptic systems with square integrable boundary data}
\author{Pascal Auscher} \author{Andreas Axelsson} \author{Alan McIntosh}
\address{Pascal Auscher, Universit\'e de Paris-Sud, UMR du CNRS 8628, 91405 Orsay Cedex, France}
\email{pascal.auscher@math.u-psud.fr}
\address{Andreas Axelsson, Matematiska institutionen, Stockholms universitet, 106 91 Stockholm, Sweden}
\email{andax@math.su.se}
\address{Alan McIntosh, Centre for Mathematics and its Applications, Mathematical Sciences Institute, Australian National University, Canberra ACT 0200, Australia}
\email{Alan.McIntosh@maths.anu.edu.au}

\mathchardef\semic="303B
\newcommand{\wedg}{\mathbin{\scriptstyle{\wedge}}}
\newcommand{\lctr}{\mathbin{\lrcorner}}
\newcommand{\R}{{\mathbf R}}
\newcommand{\C}{{\mathbf C}}

\newcommand{\mH}{{\mathcal H}}
\newcommand{\mK}{{\mathcal K}}
\newcommand{\mX}{{\mathcal X}}

\newcommand{\mL}{{\mathcal L}}

\DeclareMathOperator{\re}{Re}
\newcommand{\im}{\text{{\rm Im}}\,}
\newcommand{\sett}[2]{ \{ #1 \, \semic \, #2 \} }

\newcommand{\dist}{\text{{\rm dist}}\,}

\newcommand{\nul}{\textsf{N}}
\newcommand{\ran}{\textsf{R}}
\newcommand{\dom}{\textsf{D}}

\newcommand{\clos}[1]{\overline{#1}}
\newcommand{\conj}[1]{\overline{#1}}

\newcommand{\sgn}{\text{{\rm sgn}}}
\newcommand{\barint}{\mbox{$ave \int$}}
\newcommand{\divv}{{\text{{\rm div}}}}
\newcommand{\curl}{{\text{{\rm curl}}}}

\newcommand{\tb}[1]{\| \hspace{-0.42mm} | #1 \| \hspace{-0.42mm} |}

\newcommand{\wt}{\widetilde}
\newcommand{\ta}{{\scriptscriptstyle \parallel}}
\newcommand{\no}{{\scriptscriptstyle\perp}}
\newcommand{\pd}{\partial}

\newcommand{\oA}{{\overline A}}
\newcommand{\uA}{{\underline A}}
\newcommand{\oB}{{\overline B}}
\newcommand{\uB}{{\underline B}}

\def\barint_#1{\mathchoice
            {\mathop{\vrule width 6pt
height 3 pt depth -2.5pt
                    \kern -8.8pt
\intop}\nolimits_{#1}}%
            {\mathop{\vrule width 5pt height
3 pt depth -2.6pt
                    \kern -6.5pt
\intop}\nolimits_{#1}}%
            {\mathop{\vrule width 5pt height
3 pt depth -2.6pt
                    \kern -6pt
\intop}\nolimits_{#1}}%
            {\mathop{\vrule width 5pt height
3 pt depth -2.6pt
          \kern -6pt \intop}\nolimits_{#1}}}

\begin{document}

\begin{abstract}
We consider second order elliptic divergence form systems with complex measurable coefficients $A$
that are independent of the transversal coordinate, and prove that the set of 
$A$ for which the boundary value problem with
$L_2$ Dirichlet or Neumann data is well posed, is an open set.
Furthermore we prove that these boundary value problems are well posed when $A$
is either Hermitean, block or constant.
Our methods apply to more general systems of PDEs and as an example we prove perturbation
results for boundary value problems for differential forms.
\end{abstract}
\maketitle


\section{Introduction}

We first review the situation for scalar equations.
Consider the divergence form second order elliptic equation
\begin{equation}  \label{eq:divformscalar}
  \divv_{t,x} A(x) \nabla_{t,x}U(t,x)= \sum_{i,j=0}^n \pd_i A_{i,j}(x) \pd_j U(t,x) =0,
\end{equation}
on the upper half space $\R^{1+n}_+ := \sett{(t,x)\in\R\times \R^n}{t>0}$, $n\ge 1$,
where the matrix $A=(A_{i,j}(x))_{i,j=0}^n\in L_\infty(\R^n;\mL(\C^{1+n}))$ 
is assumed to be $t$-independent and strictly accretive with complex coefficients.
In this generality, when no regularity is assumed of the coefficients, the natural conditions
to impose on $U$ at the boundary $\R^n$ are one of the following.
\begin{itemize}
\item 
Dirichlet problem (Dir-$A$):  $U(0,x)=u(x)$ for a given function $u(x)$.
\item
Neumann problem (Neu-$A$): $-\sum_{j}A_{0,j}(x)\pd_j U(0,x)=\phi(x)$, 
where $\phi(x)$ is given.
\item Dirichlet regularity problem (Reg-$A$): $\pd_i U(0,x) = \pd_i u(x)$, $1\le i\le n$, where
$u(x)$ is given.
\end{itemize}
In this paper, we consider these boundary value problems (BVPs) in $L_2(\R^n)$, i.e.~the boundary data are $u\in L_2(\R^n)$, $\phi\in L_2(\R^n)$ and $u\in\dot H^1(\R^n)$ respectively, and
for well posedness a unique function $U(t,x)$ with certain $L_2$ estimates is required.
Detailed definitions are given in Section~\ref{sec:notres}.

These BVPs arise naturally when considering BVPs for the
Laplace equation on a Lipschitz domain $\Omega$ in $\R^n$. 
As the main problem here is a local one, the result for such domains can be derived
from the scale invariant case of a Lipschitz graph domain, i.e.~we assume that
$\Omega=\sett{(t,x)}{t>g(x)}$ is the domain above the graph of some Lipschitz function $g$.
Through a change of variables $U(t,x):= V(t+g(x), x)$, an harmonic function $V$ in $\Omega$
corresponds to $U$ in $\R^{n+1}_+$ satisfying (\ref{eq:divformscalar})
with coefficients $A=[1+|\nabla_x g|^2, -\nabla_xg^t;-\nabla_x g, I]$, and the respective
boundary conditions carry over from $\partial \Omega$ to $\R^n$.
The coefficents appearing from this pullback technique are referred to as being of {\em Jacobian type}, 
and are in particular real and symmetric, as well as independent of the transversal coordinate $t$. 
In this case, solvability of the Dirichlet problem was first proved by Dahlberg~\cite{D}, 
and solvability of the Neumann and regularity problems was first proved by
Jerison and Kenig~\cite{JK2}. Later Verchota~\cite{V} showed 
that these BVPs are solvable with the layer potential integral equation method.
For general real symmetric matrices $A$, not being of the Jacobian type, 
well posedness of the Dirichlet problem was first proved by 
Jerison and Kenig~\cite{JK1}, and the Neumann and regularity problems
by Kenig and Pipher~\cite{KP}.

It is natural to ask whether the BVPs for the fundamental elliptic equation  (\ref{eq:divformscalar})
are well posed for more general coefficients. Obvious generalizations are coefficients $A(t,x)$
with $t$-dependence, as well as more general non-symmetric or complex coefficient matrices.
In both cases, it is known that well posedness does not hold in general. 
Caffarelli, Fabes and Kenig~\cite{CaFK} observed that some regularity in the $t$-coordinate
is necessary for well posedness, and Kenig, Koch, Pipher and Toro~\cite{KKPT} and
Kenig and Rule~\cite{KR} gave examples where well posedness fail in any $L_p$ for sufficiently non-symmetric,
but $t$-independent real coefficients in the plane which are discontinuous at $x=0$.
However, on the positive side they show that, for given real non-symmetric coefficients in the plane,
the Dirichlet problem is well posed in $L_p$ for sufficiently large $p$, whereas 
the Neumann and regularity problems are well posed for $p$ sufficiently close to $1$.

In this paper, we consider only $t$-independent coefficients, but allow on the other hand 
arbitrary complex, strictly accretive coefficients $A\in L_\infty(\R^n;\mL(\C^{1+n}))$.
As remarked, well posedness does not hold in general for the BVPs in this case.
But our main result Theorem~\ref{thm:bvpfordivform} shows that the sets of well posedness
\begin{equation}   \label{eq:setsofwp}
  WP(X) := \sett{A}{\text{(X-$A$) is well posed in } L_2(\R^n)} \subset L_\infty(\R^n;\mL(\C^{1+n})),
\end{equation}
where $X$ denotes one of the three BVPs Dir, Neu or Reg,
are all open sets.
As discussed above, the sets of well posedness contain all real symmetric coefficients.
Our theorem thus in particular shows well posedness for small complex perturbations of real symmetric
coefficients.
This has also been proved in \cite{AAAHK} by Alfonseca, Auscher, Axelsson, 
Hofmann and Kim, with other methods using layer potential operators, and in
\cite{AAH} by Auscher, Axelsson and Hofmann.
In fact, our methods here give the new result that well posedness holds for
complex Hermitean matrices and their perturbations.

One may ask what is the motivation for considering complex coefficients. 
However interesting it may be to know well posedness for complex matrices, a main
motivation is that this feeds back to give perturbation estimates for real matrices.
In fact, to show that the solution $U$ varies continuously, for fixed boundary data,
as $A(x)$ varies continuously in $L_\infty$ within the subspace of 
real symmetric matrices, there is no known method which does not use bounds
for complex BVPs. 
The observation being used is that bounds for complex BVPs imply analytic dependence
on $A$ and in particular Lipschitz regularity with respect to $A\in L_\infty$.

Turning to other consequences of Theorem~\ref{thm:bvpfordivform},
well posedness is well known to hold for all constant coefficients $A(x)=A$, and
our theorem thus yields well posedness for perturbations here as well.
The Dirichlet problem was first shown to be well posed for small perturbations of constant matrices,
by Fabes, Jerison and Kenig~\cite{FJK}, using the method of multilinear 
expansions.
The Neumann and regularity problems are tackled in \cite{AAAHK} and \cite{AAH}.

It is also known that (Neu-$A$) and (Reg-$A$) are well posed in $L_2$ for complex matrices 
of block form, i.e.~such that $A_{0,i}=0=A_{i,0}$ for all $1\le i\le n$. This is a non-trivial result
and is in fact equivalent to the Kato square root problem, proved 
by Auscher, Hofmann, Lacey, McIntosh and Tchamitchian~\cite{AHLMcT}.
Our theorem thus yields well posedness for small perturbations of complex block
form matrices, and is in this sense a generalization of the Kato square root estimate.
With the further assumption of pointwise resolvent kernel bounds, this result
is also implicit in \cite{AAAHK}.  
However, our methods in this paper require no such pointwise estimates.

Let us now discuss the methods underlying Theorem~\ref{thm:bvpfordivform}.
For the proof we use, following \cite{AAH}, boundary equation methods
involving a Cauchy operator $E_A$.
The name Cauchy operator is used since $E_A$ coincides with the Cauchy singular integral operator 
when $A=I$ and $n=1$.
The first step of the proof is to rewrite the second order equation (\ref{eq:divformscalar}) as the 
equivalent first order system
\begin{align} \label{eq:Laplacein1order}
\begin{cases}   
  \divv_{t,x} A(x) F(t,x)  =0, \\
  \curl_{t,x} F(t,x) =0,
\end{cases}
\end{align} 
taking the gradient vector field $F(t,x):= \nabla_{t,x}U(t,x)$ as the unknown function instead
of the potential $U$.
The Cauchy operator $E_A$ is related to (\ref{eq:Laplacein1order}) in the same way that 
the classical Cauchy integral operator is related to the Cauchy--Riemann equations.
Just as the Cauchy singular integral operator is a Fourier multiplier with symbol $\sgn(\xi)$, that is,  belongs
to the functional calculus of $d/dx$, the
operator $E_A$ belongs to the functional calculus of
a first order differential operator $T_A$.
The bisectorial operator $-T_A$ in $L_2(\R^n;\C^{1+n})$ is the infinitesimal generator for the system 
(\ref{eq:Laplacein1order}) in the sense that these equations are equivalent to
$\pd_t F+ T_A F=0$.
The fundamental problem is to prove that this operator $T_A$ has a bounded holomorphic 
functional calculus, and as a consequence that the Cauchy operator $E_A$ is bounded.
Given this, the perturbation results for BVPs follow as a consequence.

In \cite{AAH}, it was proved that $\|E_A\|<\infty$
when $\|A-A_0\|_\infty<\epsilon$ and $A_0$ is either real symmetric, block or constant.
This paper made use of a rather lengthy perturbation argument, and also used 
square function estimates of Dahlberg, Jerison and Kenig~\cite{DJK} and estimates of harmonic measure
of Jerison and Kenig~\cite{JK1}, for 
solutions to (\ref{eq:divformscalar}) in the real symmetric case.

In this paper we prove the boundedness of the holomorphic functional calculus of $T_A$,
for all complex $A$, directly from the quadratic estimates proved by Axelsson, Keith and McIntosh
\cite{AKMc}. In this way, our results build on the proof of the Kato square root problem \cite{AHLMcT}.
That $\|E_A\|<\infty$ for all complex $A$ may come as a surprise, in view of the above mentioned
counter-examples to well posedness of the BVPs for non-symmetric coefficients.
However, it is important to note that $E_A$ itself has nothing to do with BVPs, it is an
infinitesimal generator associated with the differential equation, and is not related
to the boundary conditions (except in the case  of block form matrices).
As a consequence of the boundedness of $E_A=\sgn(T_A)$, we prove in Theorem~\ref{thm:mainhardy}
that there is a Hardy type splitting
$$
   L_2(\R^n;\C^{1+n}) \ni f = F^+|_{\R^n} + F^-|_{\R^n}
$$
of boundary functions $f$ into traces of $F^\pm$ satisfying 
(\ref{eq:Laplacein1order}) in $\R^{1+n}_\pm$, with estimates 
$\|f\|_2 \approx \|F^+|_{\R^n}\|_2 + \|F^-|_{\R^n}\|_2$.
That a BVP is well posed is the question whether the full traces $F^+|_{\R^n}$
of solutions to the equations in $\R^{n+1}_+$ are in one-to-one correpondence
with the normal components (Neumann problem) or tangential parts (Dirichlet regularity problem)
respectively.
It is this one-to-one correspondence which may fail for some complex $A$.
What we prove here and use for the proof of Theorem~\ref{thm:bvpfordivform}, is that
the Hardy subspaces $\{F^+|_{\R^n}\}$, and the projections $f\mapsto F^+|_{\R^n}$ onto them,
depend analytically on $A$.

Finally we note that the methods developed here go beyond scalar elliptic
equations like (\ref{eq:divformscalar}). 
The natural framework here is rather BVPs for elliptic systems of partial differential equations,
as should be clear from (\ref{eq:Laplacein1order}).
Thus we shall formulate our results for divergence form elliptic systems of $m$ second order equations 
(but the reader interested in scalar equations only can set $m=1$ throughout). 
In this setting, our well posedness results are mostly new. 
Previously known results are limited to systems with coefficients of Jacobian type, or more generally 
constant coefficient systems on Lipschitz domains.

Well posedness of (Dir-$A$), (Reg-$A$) and (Neu-$A$) with $L_2$ boundary values have been obtained for 
the Stokes' system by Fabes, Kenig and Verchota~\cite{FKV}, and of (Dir-$A$) and (Reg-$A$) for 
the Lam\'e system by Dahlberg, Kenig and Verchota~\cite{DKV}.
For general constant coefficient symmetric second order systems, 
solvability result for (Neu-$A$) and (Reg-$A$) are found in Fabes~\cite{F}.
Under the weaker Legendre--Hadamard ellipticity condition, (Dir-$A$) and (Reg-$A$)
where solved by Gao~\cite{Gao}.
As for non-symmetic systems, Verchota and Vogel~\cite{VV97}  obtained $L_p$ solvability
results for (Dir-$A$), (Reg-$A$) and (Neu-$A$) in the spirit of \cite{KKPT} and \cite{KR} 
for certain non-symmetric constant coefficient Legendre--Hadamard systems of two equations on 
$C^1$ polygons in the plane.
For general elliptic systems, the Kato problem was solved by
Auscher, Hofmann, McIntosh and Tchamitchian~\cite{AHMcT}, a consequence being the
well posedness of (Neu-$A$) and (Reg-$A$) for elliptic systems with block form coefficients.

Note that the pullback technique, from the Lipschitz domain $\Omega$ to $\R^{n+1}_+$, described above, works for more general divergence form equations or systems. 
In this case, coefficients $\tilde A(x)$ in $\Omega$ are transformed into coefficients
$$
  A(x) := 
     \begin{bmatrix}
       1 & -(\nabla_x g(x))^t \\
       0 & I
     \end{bmatrix}
     \wt A(x)
     \begin{bmatrix}
       1 & 0 \\
       -\nabla_x g(x) & I
     \end{bmatrix}
$$
in $\R^{n+1}_+$.
We also remark that our methods are by no means limited to divergence form equations.
In Section~\ref{sec:forms}, we give solvability results for exterior differential systems 
(\ref{eq:diracwedgek}) for differential forms, as an example of this.
This generalizes the first order system (\ref{eq:Laplacein1order}), which is the special case
of (\ref{eq:diracwedgek}) for $1$-forms.
Furthermore, we note that time-harmonic Maxwell's equations on a Lipschitz domain
can be written as a system of equations (\ref{eq:diracwedgek}) for $1$ and $2$-forms with 
lower order terms added, through the above pullback technique.
Thus, although not directly applicable, (\ref{eq:diracwedgek}) is closely related
to Maxwell's equations.
Solvability of Maxwell's equations on Lipschitz domains is due to Mitrea~\cite{Mit},
and more general BVPs for Dirac equations were solved by McIntosh and Mitrea~\cite{McMi}.
In fact, the Cauchy integral boundary equation method used in this paper, as well as 
in \cite{AAH},  was developed for solving BVPs for Maxwell's and Dirac's equations
in the PhD thesis \cite{Axthesis} of the second author. 
Further elaborations of the ideas presented in this paper, along the lines of thought in \cite{AAH}
and \cite{AKMcMix},
working with general inhomogeneous differential forms taking values in the full exterior
algebra and allowing lower order terms, one should be able to extend the theory to 
cover both Dirac's and Maxwell's equations.

{\bf Acknowledgments.}

This work is an outgrowth of joint research conducted at Stockholm University and at Universit\'e de Paris-Sud, Orsay.
We acknowledge partial support from PHC FAST d'EGIDE No.12739WA 
for a trip of Auscher to Canberra.
Axelsson was supported by a travel grant from the Swedish Research Council, and by Universit\'e de Paris-Sud.
McIntosh acknowledges support from Stockholm University, Universit\'e de Paris-Sud, as well as the Centre for Mathematics and its Applications at the Australian National
University. This research was also supported by the Australian Government through the Australian Research  Council and through the International Science Linkages FAST program.

The authors acknowledge Steve Hofmann for freely sharing his insight into the harmonic analysis underlying this paper, and for the suggestion that the results be presented for systems rather than for single equations.

\section{Notation and results}   \label{sec:notres}

We begin by giving the precise definition of well posedness of the BVPs discussed in the 
introduction, or rather the corresponding BVPs for systems.
Throughout this paper, we use the notation $X \approx Y$ and $X \lesssim Y$ for estimates
to mean that there exists a constant $C>0$, independent of the variables in the estimate, 
such that $ X/C \le Y \le CX$ and $X\le C Y$, respectively.

We write $\{e_0,e_1,\ldots,e_n\}$ for the standard basis for $\R^{1+n}$ with 
$e_0$ ``upward'' pointing into $\R^{1+n}_+$, and write $t=x_0$ for the vertical
coordinate. For the vertical derivative, we write $\partial_0 = \partial_t$.
For vectors $v=(v_i^\alpha)_{0\le i\le n}^{1\le \alpha\le m}$, 
we write $v_0$ and $v_\ta$ for the
normal and tangential parts of $v$, 
i.e.~$(v_0)_0^\alpha= v_0^\alpha$ and $(v_0)_i^\alpha=0$ when $1\le i\le n$,
whereas
$(v_\ta)_i^\alpha= v_i^\alpha$ when $1\le i\le n$ and $(v_\ta)_0^\alpha=0$.
Frequently, we shall identify normal vector fields $v= v_0$ with the corresponding
scalar functions $v_0=(v_0^\alpha)_{\alpha=1}^m$.

For systems, gradient and divergence act as $(\nabla_{t,x}U)_i^\alpha= \pd_i U^\alpha$ and 
$(\divv_{t,x}F)^\alpha= \sum_{i=0}^n \pd_i F^\alpha_i$, with correponding
tangential versions $\nabla_x U= (\nabla_{t,x}U)_\ta$
and $(\divv_x F)^\alpha= \sum_{i=1}^n \pd_i F^\alpha_i$.
With $\curl_{t,x} F=0$ we understand that 
$\pd_jF_i^\alpha= \pd_i F_j^\alpha$, for all $i$, $j=0,\ldots,n,\alpha=1,\ldots, m$.
Similarly, write $\curl_{x} F_\ta=0$ if 
$\pd_jF_i^\alpha= \pd_i F_j^\alpha$, for all $i$, $j=1,\ldots,n,\alpha=1,\ldots, m$.

We consider divergence form second order elliptic systems
\begin{equation}  \label{eq:divform}
  \sum_{i,j=0}^n \sum_{\beta=1}^m \pd_i A^{\alpha,\beta}_{i,j}(x) \pd_j U^\beta(t,x) =0,\qquad \alpha=1,\ldots, m,
\end{equation}
on the half spaces $\R^{1+n}_\pm := \sett{(t,x)\in\R\times \R^n}{\pm t>0}$, $n\ge 1$,
where the matrix $A=(A^{\alpha,\beta}_{ij}(x))_{i,j=0,\ldots ,n}^{\alpha,\beta=1,\ldots ,m}\in L_\infty(\R^n;\mL(\C^{(1+n)m}))$ 
is assumed to be $t$-independent with complex coefficients and {\em strictly accretive} on
$\nul(\curl_\ta)$, in the sense that
there exists $\kappa>0$ such that
\begin{equation}   \label{eq:accrassumption}
  \sum_{i,j=0}^n\sum_{\alpha,\beta=1}^m \int_{\R^n} \re (A_{i,j}^{\alpha,\beta}(x)f_j^\beta(x) \conj{f_i^\alpha(x)}) dx\ge \kappa 
   \sum_{i=0}^n\sum_{\alpha=1}^m \int_{\R^n} |f_i^\alpha(x)|^2dx,
\end{equation}
for all 
$f\in\nul(\curl_\ta):= \sett{g\in L_2(\R^n;\C^{(1+n)m})}{\curl_x(g_\ta)=0}$. 

Equivalently, this means that a G\aa rding inequality
$$
 \int_{\R^n} \re \Big( A \big(\divv_x u_\ta + \nabla_x u_0 \big), \big( \divv_x u_\ta + \nabla_x u_0\big) \Big) dx \\
 \ge \kappa \int_{\R^n} (|\divv_x u_\ta|^2 + |\nabla_x u_0|^2) dx
$$
holds for $u= [u_0, u_\ta]\in \dot H^1(\R^n; \C^{(1+n)m})$, since $\divv_x$ has dense range in $L_2(\R^n;\C^m)$
and $\nabla_x$ has dense range in $\sett{g\in L_2(\R^n;\C^{nm})}{\curl_x g=0}$.
Splitting $\C^{(1+n)m}$ into normal parts $\C^m$ and tangential parts $\C^{nm}$, we write
$$
  A(x)v=
  \begin{bmatrix}
     A_{00}(x) & A_{0\ta}(x) \\ A_{\ta 0}(x) & A_{\ta\ta}(x) 
  \end{bmatrix}
  \begin{bmatrix} v_0 \\ v_\ta \end{bmatrix}.
$$
It is then clear that (\ref{eq:accrassumption}) implies that $A_{00}$ is pointwise strictly accretive
and that $A_{\ta\ta}$ satisfies a strict G\aa rding inequality
\begin{equation}   \label{eq:blockaccr}
  \begin{cases}
    \re(A_{00}(x)v,v)\ge \kappa |v|^2, & \quad v\in \C^{m}, \text{ a.e. } x\in\R^n,  \\
    \int_{\R^n}\re (A_{\ta\ta}\nabla_{x}u_0, \nabla_{x} u_0)dx \ge \kappa \int_{\R^n} |\nabla_{x} u_0|^2 dx,
    & \quad u_0\in \dot H^1(\R^n;\C^{m}).
  \end{cases}
\end{equation}
The condition (\ref{eq:accrassumption}) lies between {\em pointwise strict accretivity},
i.e. 
\begin{equation}    \label{eq:strongaccr}
  \re(A(x)v,v)\ge \kappa |v|^2,\qquad \text{for all } v\in \C^{(1+n)m} \text{ and a.e. } x\in\R^n, 
\end{equation}
and the $\R^{n+1}_+$ G\aa rding inequality
\begin{equation}   \label{eq:uppergarding}
  \iint_{\R^{n+1}_+} \re(A(x)\nabla_{t,x}g(t,x), \nabla_{t,x} g(t,x)) dtdx\ge \kappa \iint_{\R^{n+1}_+} |\nabla_{t,x} g(t,x)|^2 dtdx
\end{equation}
for $g\in \dot H^1(\R^{n+1}_+;\C^m)$.
Clearly, (\ref{eq:strongaccr}) implies (\ref{eq:accrassumption}), which in turn implies
(\ref{eq:uppergarding}), as is seen by taking $f(x)= \nabla_{t,x} g(t,x)$ for fixed $t$ and then 
integrating over $t$.
Furthermore (\ref{eq:uppergarding}), implies (\ref{eq:blockaccr}), which is seen by taking
 $g(t,x):= \psi(\epsilon t)u_0(x)$ and integrating away $t$, for some $\psi\in C^\infty_0(\R_+)$.
 Letting $\epsilon\rightarrow \infty$ and $\epsilon\rightarrow 0$ respectively proves (\ref{eq:blockaccr}).
In fact, only the G\aa rding inequality (\ref{eq:uppergarding}) for $g\in H^1_0(\R^{n+1}_+;\C^m)$
is needed for this argument.

When $n=1$, (\ref{eq:accrassumption}) is equivalent to strong accretivity (\ref{eq:strongaccr}) 
since in this case $\nabla_x$ has dense range in $L_2(\R;\C^m)$ and $\nul(\curl_\ta)=L_2(\R;\C^{2m})$.
On the other hand, if $A$ is of block form, i.e. $A_{0\ta}=A_{\ta 0}=0$, then (\ref{eq:accrassumption})
is equivalent to (\ref{eq:blockaccr}) and to the $\R^{n+1}_+$ G\aa rding inequality (\ref{eq:uppergarding}),
for $H^1(\R^{n+1}_+;\C^m)$ as well as for $H^1_0(\R^{n+1}_+;\C^m)$,
since (\ref{eq:blockaccr}) implies (\ref{eq:accrassumption}).
It is also known that the $\R^{n+1}_+$ G\aa rding inequality (\ref{eq:uppergarding}) implies 
strong accretivity (\ref{eq:strongaccr}) when $m=1$, 
so for scalar equations (\ref{eq:accrassumption}), (\ref{eq:uppergarding}) and (\ref{eq:strongaccr})
are all equivalent.

On the functions $U=(U^\alpha)_{\alpha=1}^m$ satisfying (\ref{eq:divform}), 
we impose one of the following boundary conditions.
\begin{itemize}
\item 
  (Dir-$A$):  $U^\alpha(0,x)=u^\alpha(x)$ for a given function $u\in L_2(\R^n;\C^m)$.
\item
 (Neu-$A$): $-\sum_{j,\beta}A_{0,j}^{\alpha,\beta}(x)\pd_j U^\beta (0,x)=\phi^\alpha(x)$, 
where $\phi\in L_2(\R^n;\C^m)$ is given.
\item (Reg-$A$): $\pd_i U^\alpha(0,x) = \pd_i u^\alpha(x)$, $1\le i\le n$, where
$u\in\dot H^1(\R^n;\C^m)$ is given.
\end{itemize}

The boundary value problems (Neu-$A$) and (Reg-$A$) 
can be viewed as problems concerning a first order partial
differential system, and this is the point of view we take here.
Indeed, consider the gradient vector fields
$$
  F(t,x)=\nabla_{t,x}U(t,x): \R^{1+n}_+\longrightarrow \C^{(1+n)m}.
$$
Since the scalar potentials $U$ are in one-to-one correspondence with the curl-free vector
fields $F(t,x)$, modulo constants, we can take $F$ rather than $U$ as the unknown,
and equation (\ref{eq:divform}) for $U$ is rewritten
as the equivalent first order system (\ref{eq:Laplacein1order}) for $F$.
Since the coefficients $A(x)$ are independent of $t$, it is natural to view $F$ 
from the semigroup point of view 
$F(t,x)= F_t(x)= \nabla_{t,x}U(t,x) \in C^1(\R_+; L_2(\R^n;\C^{(1+n)m}))$.

\begin{defn}   \label{defn:wellposed}
\begin{itemize}
\item[{\rm (i)}]
  We say that the boundary value problem (Neu-$A$) is {\em well posed} if for each 
  boundary data $\phi\in L_2(\R^n;\C^m)$, there exists a unique function
$$
  F(t,x)= F_t(x)= \nabla_{t,x}U(t,x) \in C^1(\R_+; L_2(\R^n;\C^{(1+n)m}))
$$
which satisfies (\ref{eq:Laplacein1order}) for $t>0$, and has
limits $\lim_{t\rightarrow \infty}F_t=0$ and $\lim_{t\rightarrow 0}F_t=f$ in $L_2$ norm, 
where the full boundary trace $f$ satisfies the boundary condition $-(Af)_0=\phi$.
More precisely, by $F$ satisfying (\ref{eq:Laplacein1order}), we mean that 
$\pd_t (AF)_0=-\divv_x(AF)_\ta$, $\pd_t F_\ta= \nabla_x F_0$ and $\curl_x F_\ta=0$, 
where $\pd_t$ is taken in the strong sense, and $x$-derivatives
in the distributional sense.
\item[{\rm (ii)}]
  We say that the boundary value problem (Reg-$A$)
  is {\em well posed} if for each boundary data $\nabla_x u\in L_2(\R^n;\C^{nm})$, 
there exists a unique function
$F\in C^1(\R_+; L_2(\R^n;\C^{(1+n)m}))$
which satisfies (\ref{eq:Laplacein1order}) for $t>0$, and has
limits $\lim_{t\rightarrow \infty}F_t=0$ and $\lim_{t\rightarrow 0}F_t=f$
in $L_2$ norm, where the full boundary trace $f$ satisfies the boundary condition $f_\ta=\nabla_x u$.
\item[{\rm (iii)}]
The Dirichlet problem (Dir-$A$) is said to be {\em well posed} if for each
$u\in L_2(\R^n;\C^m)$, there is a unique function
$$
U_t(x) =U(t,x)\in C^1(\R_+; L_2(\R^n;\C^m))
$$
such that $\nabla_{x} U\in C^0(\R_+; L_2(\R^n;\C^{nm}))$,
where $U$ satisfies (\ref{eq:divform}) for $t>0$,
$\lim_{t\rightarrow 0}U_t=u$,
$\lim_{t\rightarrow \infty}U_t=0$, $\lim_{t\rightarrow \infty}\nabla_{t,x} U_t=0$ 
in $L_2$ norm,
and 
$\int_{t_0}^{t_1} \nabla_{x} U_s\, ds$
converges in $L_2$ when $t_0\rightarrow 0$ and $t_1\rightarrow\infty$.
More precisely, by $U$ satisfying (\ref{eq:divform}), we mean that 
$\int_t^\infty ((A\nabla_{s,x}U_s)_\ta, \nabla_x v)ds= -( (A\nabla_{t,x}U_t)_0, v)$
for all $v\in C_0^\infty(\R^{n};\C^m)$.
\end{itemize}
\end{defn}

The Dirichlet problem (Dir-$A$) will also be rewritten as a BVP for the first order system
(\ref{eq:Laplacein1order}).
However, here it is not appropriate to consider the gradient vector field
$\nabla_{t,x}U$, since the boundary condition is a condition on the potential $U$ itself.
Instead we use the point of view of harmonic conjugate functions, and write
$F= U e_0 + F_\ta$, where the tangential vector fields $F_\ta$ are conjugate functions in a 
generalized sense and $F$ satisfies (\ref{eq:Laplacein1order}), which is viewed as
a generalized Cauchy--Riemann system.
Details of this are given in Lemma~\ref{lem:conjfcnreduction}, where it is shown that 
the Dirichlet problem (Dir-$A$) for $U$
is equivalent to an auxiliary Neumann problem (Neu$^\perp$-$A$) for $F$.

Our main result, which we prove this in Section~\ref{sec:divbvp}, is the following.
\begin{thm}   \label{thm:bvpfordivform}
  The sets $WP(Reg)$, $WP(Neu)$ and $WP(Dir)$, as defined in (\ref{eq:setsofwp}), are all open
  subsets of $L_\infty(\R^n; \mL(\C^{(1+n)m}))$.
  Each of the sets of well posedness contains
\begin{itemize}
\item[{\rm (i)}] all Hermitean matrices $A(x)= A(x)^*$ (and in particular all real symmetric matrices),
\item[{\rm (ii)}] all block matrices where $A_{0,i}^{\alpha,\beta}(x)=0=A_{i,0}^{\alpha,\beta}(x)$, $1\le i\le n, 1\le \alpha, \beta\le m$, and
\item[{\rm (iii)}] all constant matrices $A(x)=A$.
\end{itemize}
\end{thm}

The notion of well posedness used here departs from the standard variational one.
However, we show in Section~\ref{sec:unique} that the solutions obtained 
here coincide with the solutions obtained through the Lax--Milgram Theorem
when $A$ belongs to the connected component of $WP$ which contains $I$. 
This connected component includes the three classes (i), (ii) and (iii)
specified in Theorem~\ref{thm:bvpfordivform}.
The notion of well posedness used here coincides with that in \cite{AAH} for the BVPs (Neu-$A$) and (Reg-$A$) .
However, for (Dir-$A$)  the meaning of well posedness differs slightly from that in \cite{AAH},
as we impose an extra integrability condition on $\nabla_x U$ here.

A natural function space for solutions $F_t(x) \in C^1(\R_+; L_2(\R^n;\C^{(1+n)m}))$ to the BVPs
is $L_\infty(\R_+;L_2)$, with norm $\sup_{t>0} \|F_t\|_2$.
Two other norms which are important are the square function norm $\tb{F_t}^2:= \int_0^\infty \|F_t\|_2^2\, t^{-1}dt$
and the norm $\|\widetilde N_* (F)\|_2$, using the modified {\em non-tangential maximal function}
$$
\widetilde N_*(F)(x):= \sup_{t>0}  t^{-(1+n)/2} \|F\|_{L_2(Q(t,x))},
$$
where $Q(t,x):= [(1-c_0)t,(1+c_0)t]\times B(x;c_1t)$,
for some fixed constants $c_0\in(0,1)$, $c_1>0$.

The key result underlying Theorem~\ref{thm:bvpfordivform}, which we prove in Section~\ref{sec:ahat},
is the following result on 
Hardy type splittings of $L_2(\R^n)$.
\begin{thm}  \label{thm:mainhardy}
Let $A \in L_\infty(\R^n;\mL(\C^{(1+n)m}))$ be a $t$-independent, complex matrix function
which is strictly accretive on $\nul(\curl_\ta)$.

Then each $f\in\nul(\curl_\ta)$ is in one-to-one correspondence 
with a pair of vector
fields $F^\pm_t(x)=F^\pm(t,x)\in C^1(\R_\pm; L_2(\R^n;\C^{(1+n)m}))$ in $\R^{1+n}_\pm$
satisfying (\ref{eq:Laplacein1order}) and having $L_2$ limits 
$\lim_{t\rightarrow 0^\pm}F^\pm_t=f^\pm$ and $\lim_{t\rightarrow \pm\infty}F^\pm_t=0$,
such that  
$$
f=f^++f^-.
$$
This splitting is topological, i.e.~$\|f\|_2\approx \|f^+\|_2 +\|f^-\|_2$,
and the vector fields $F^\pm$ satisfy norm equivalences
$$
   \|f\|_2\approx \sup_{t>0} \|F_t\|_2
  \approx \tb{t\pd_t F_t}
  \approx \|\widetilde N_* (F)\|_2.
$$
Moreover, the Hardy projections $E_A^\pm:f\mapsto F^\pm=F^\pm_A$
depend locally Lipschitz continuously on $A$ in the sense that
$$
  \|F^\pm_{A_2}-F^\pm_{A_1}\|_{\mX}\le C \|A_2-A_1\|_{L_\infty(\R^n)} \|f\|_2,
$$
where $C=C(\kappa_{A_1}, \kappa_{A_2}, \|A_1\|_\infty, \|A_2\|_\infty)$ and where
$\|F\|_\mX$ denotes any of the four norms above.
\end{thm}

Restricting our attention to the Dirichlet problem in the upper half space, 
Theorem~\ref{thm:mainhardy} shows in particular the following.

\begin{cor}   \label{cor:DJK}
Let $A \in L_\infty(\R^n;\mL(\C^{(1+n)m}))$ be a $t$-independent, complex matrix function
which is strictly accretive on $\nul(\curl_\ta)$ and 
assume that $A\in WP(Dir)$.
Then any function $U_t(x) =U(t,x)\in C^1(\R_+; L_2(\R^n;\C^m))$ solving (\ref{eq:divform}),
with properties as in Definition~\ref{defn:wellposed}, has estimates
$$
  \int_{\R^n}|u|^2 dx\approx\sup_{t>0}\int_{\R^n}|U_t|^2 dx\approx \int_{\R^n}  | \wt N_*(U)|^2 dx 
  \approx \iint_{\R^{1+n}_+} |\nabla_{t,x} U|^2 t \, dtdx,
$$
where $u=U|_{\R^n}$.
If furthermore $A$ is real (not necessarily symmetric) and $m=1$, then 
Moser's local boundedness  estimate \cite{Mos} gives the pointwise estimate
$\wt N_*(U)(x) \approx N_*(U)(x)$, where the standard non-tangential maximal function is
$N_*(U)(x):= \sup_{|y-x|<c t}|U(t,y)|$,
for fixed $0<c<\infty$.
\end{cor}
Theorem~\ref{thm:bvpfordivform} shows in particular that $A\in WP(Dir)$ if $A$ is real symmetric.
Even for real symmetric $A$ our methods yield a new proof of the estimate between 
the square function and the non-tangential maximal function above,
first proved by Dahlberg, Jerison and Kenig~\cite{DJK},
using estimates of harmonic measure by Jerison and Kenig~\cite{JK1},
in the scalar case $m=1$.
In \cite{AAH}, these estimates were used to prove $\|E_A\|<\infty$ for real symmetric $A$. 
Here we reverse the argument: we prove $\|E_A\|<\infty$ independently and 
deduce from this the equivalences of norms.

In the case $m>1$ of systems, Dahlberg, Kenig, Pipher and Verchota~\cite{DKPV} have shown
equivalence in $L_p$ norm of the non-tangential maximal function and the square function,
for general constant coefficient real symmetric systems on Lipschitz domains.
Also, such equivalence has been shown for certain non-symmetric systems of two equations 
in the plane by Verchota and Vogel~\cite{VV00}.

\section{Cauchy operators and Hardy spaces}   \label{sec:ahat}

The boundary equation methods for solving BVPs which are used here and in \cite{AAH},
are based on Cauchy operators $E_A$, 
with associated Hardy type subspaces.
In this section we prove quadratic estimates for $E_A$ and deduce from this Theorem~\ref{thm:mainhardy}.
How these Cauchy operators are used to prove Theorem~\ref{thm:bvpfordivform}
is shown in Section~\ref{sec:divbvp} and in particular Lemma~\ref{lem:vardomains}.
Note that the operators $E_A$ themselves depend only on the differential system (\ref{eq:Laplacein1order}),
and have nothing to do with the boundary conditions. Thus they are the same for both Neumann and Dirichlet problems.

We start by rewriting the Equations (\ref{eq:Laplacein1order}) in terms of an
``infinitesimal generator'' $T_A$.
Write $v\in \C^{(1+n)m}$ as $v = [v_0, v_\ta]^t$, where $v_0\in\C^m$ and $v_\ta\in\C^{nm}$, and
introduce the auxiliary matrices 
$$
  \oA:=
  \begin{bmatrix}
     A_{00} & A_{0\ta} \\ 0 & I
  \end{bmatrix}, \quad
  \uA:=
  \begin{bmatrix}
     1 & 0 \\ A_{\ta 0} & A_{\ta\ta} 
  \end{bmatrix},
  \qquad\text{if }
  A=
  \begin{bmatrix}
     A_{00} & A_{0\ta} \\ A_{\ta 0} & A_{\ta\ta} 
  \end{bmatrix}
$$
in the normal/tangential splitting of $\C^{(1+n)m}$.
Recall that $A$ being strictly accretive on $\nul(\curl_\ta)$, as in (\ref{eq:accrassumption}),
implies the accretivity estimates (\ref{eq:blockaccr}) for the diagonal blocks $A_{00}$
and $A_{\ta\ta}$.
Since $A_{00}$ is pointwise strictly accretive, it is invertible, and consequently $\oA$ is invertible.
This is not necessarily true for $\uA$.

Splitting normal and tangential derivatives in (\ref{eq:Laplacein1order}), we see that 
this system of equations is equivalent to
$$
\begin{cases}   
  \pd_t ( A F)_0 + \divv_x(AF)_\ta  =0, \\
  \pd_t F_\ta-\nabla_x F_0 =0,
\end{cases}
$$
together with the constraint $\curl_x F_\ta=0$. Since
\begin{equation}   \label{eq:uoa}
   \oA F=(AF)_0+F_\ta \quad\text{and} \quad
   \uA F= F_0+(AF)_\ta, 
\end{equation}  
we have shown that
(\ref{eq:Laplacein1order}) is equivalent to 
$$
  \begin{cases}
    \pd_t (\oA F)_0 + \divv_x(\uA F)_\ta = 0, \\
    \pd_t (\oA F)_\ta -\nabla_x(\uA F)_0=0,
  \end{cases}
$$
together with the tangential constraint $\curl_x F_\ta=0$.
Combining the two equations, we get
\begin{equation}   \label{eq:generator}
  \pd_t F+ T_A F=0,
\end{equation}
where $T_A$ is the following operator.

\begin{defn}  \label{defn:hattransforms}

Let $D$ be the self-adjoint differential operator 
$D:=    \begin{bmatrix}
        0 & \divv_x \\
        -\nabla_x & 0
     \end{bmatrix}
$
in $L_2(\R^n;\C^{(1+n)m})$ with 
domain $\dom(D):= \sett{[f_0, f_\ta]^t\in L_2(\R^n;\C^{(1+n)m})}{\nabla_x f_0, \, \divv_x f_\ta\in L_2}$.
Define the infinitesimal generator for (\ref{eq:Laplacein1order}) to be the
operator
$$
  T_A:= \oA^{-1} D \uA=      
     \begin{bmatrix}
        A_{00}^{-1}(A_{0\ta}\nabla_x+\divv_x A_{\ta 0} ) &
        A_{00}^{-1}\divv_x A_{\ta\ta} \\
        -\nabla_x & 0
     \end{bmatrix},
$$
with domain $\dom(T_A):= \uA^{-1}\dom(D)$.
Let the transformed coefficient matrix be
$$
   \hat A:= \uA \oA^{-1}= 
  \begin{bmatrix}
     A_{00}^{-1} & - A_{00}^{-1 }A_{0\ta} \\ A_{\ta 0} A_{00}^{-1} & A_{\ta\ta} -A_{\ta 0} A_{00}^{-1} A_{0\ta} 
  \end{bmatrix}
$$
so that
\begin{equation}   \label{eq:similarity}
   T_A= \oA^{-1}(D\hat A)\oA. 
\end{equation}
\end{defn}

The following is the main algebraic result of the paper.
Recall that $\nul(\curl_\ta)= \sett{g\in L_2(\R^n;\C^{(1+n)m})}{\curl_x(g_\ta)=0}$
and note that $\nul(\curl_\ta)=\clos{\ran(D)}$. 
\begin{prop}  \label{prop:ahataccr}
   The transformed coefficient matrix 
   $\hat A$ is bounded and strictly accretive on $\nul(\curl_\ta)$, i.e. satisfies (\ref{eq:accrassumption}),
    if and only if $A$ is bounded and strictly accretive on $\nul(\curl_\ta)$.
   Moreover $\hat{\hat A}=A$.
\end{prop}
\begin{proof}
 Assume that $A$ is bounded and strictly accretive on $\nul(\curl_\ta)$.
 As noted above, $\oA$ is invertible and thus $\hat A$ is bounded. 
 Since $\oA$ acts as identity on tangential vector fields, it is clear that
 $\oA:\nul(\curl_\ta)\rightarrow\nul(\curl_\ta)$ is an isomorphism.
 Strict accretivity of $\hat A$ on $\nul(\curl_\ta)$ now follows from the formula
\begin{multline*}
\re(\hat A (\oA f),\oA f) = \re(\uA f,\oA f) = \re\big( (f_0, A_{00} f_0 + A_{0\ta} f_\ta) + (A_{\ta 0}f_0+ A_{\ta\ta}f_\ta, f_\ta) \big) \\
= \re\big( (A_{00} f_0 + A_{0\ta} f_\ta, f_0) + (A_{\ta 0}f_0+ A_{\ta\ta}f_\ta, f_\ta) \big) = \re (A f,f).
\end{multline*}
The identity $\hat{\hat A}=A$ is straightforward to verify, and this shows that the above 
argument is reversible.
\end{proof}
We are now in a position to analyze the operator $T_A$. Due to (\ref{eq:similarity}),
it suffices to study operators of the form $DB$ in $L_2(\R^n; \C^N)$, where 
$D$ is a self-adjoint homogeneous first order differential operator with constant coefficients,
and $B$ is a bounded multiplication operator which is strictly accretive on
$\ran(D)$, i.e.~there exists $\kappa>0$ such that
\begin{equation}  \label{eq:accronrange}
   \re(BDu,Du)\ge \kappa \|Du\|_2,\qquad\text{for all } u\in \dom(D).
\end{equation}
The applications we have in mind are the specific operators $D$ and $B=\hat A$ from 
Definition~\ref{defn:hattransforms}, in which case $\clos{\ran(D)}= \nul(\curl_\ta)$,
as well as generalizations of these 
in Section~\ref{sec:forms}.

Define closed and open sectors and double sectors in the complex plane by
\begin{alignat*}{2}
    S_{\omega+} &:= \sett{z\in\C}{|\arg z|\le\omega}\cup\{0\}, 
    & \qquad
    S_{\omega} &:= S_{\omega+}\cup(-S_{\omega+}), \\    
    S_{\nu+}^o &:= \sett{z\in\C}{ z\ne 0, \, |\arg z|<\nu},  
    & \qquad  
    S_{\nu}^o &:= S_{\nu+}^o\cup(- S_{\nu+}^o),
\end{alignat*}
and define the {\em angle of accretivity} of $B$ to be
$$
   \omega:= \sup_{f\not = 0,f\in\ran(D)} |\arg(Bf,f)|  <\pi/2.
$$
\begin{prop}   \label{prop:typeomega}
Let $D$ be a self-adjoint operator and let $B$ be a bounded operator in $L_2(\R^n; \C^N)$
which satisfies (\ref{eq:accronrange}).
\begin{itemize}
\item[{\rm (i)}]
The operator $DB$ is a closed and densely defined $\omega$-bisectorial operator, i.e.~$\sigma(DB)\subset S_\omega$, where
$\omega$ is the angle of accretivity of $B$.
Moreover, there are resolvent bounds 
$\|(\lambda I - DB)^{-1}\| \lesssim 1/ \dist(\lambda, S_\omega)$ when $\lambda\notin S_\omega$.
\item[{\rm (ii)}]
The operator $DB$ has range $\ran(DB)=\ran(D)$ and null space $\nul(DB)=B^{-1}\nul(D)$,
where $B^{-1}$ denotes the inverse image, such that
$$
   L_2(\R^n; \C^N) = \clos{\ran(DB)} \oplus \nul(DB)
$$
topologically (but in general non-orthogonally).
\item[{\rm (iii)}]
The restriction of $DB$ to $\clos{\ran(D)}=\clos{\ran(DB)}$ 
is a closed and injective operator with dense range in
$\clos{\ran(D)}$, with estimates on spectrum and resolvents as in (i).
\end{itemize}
\end{prop}

\begin{proof}
   As a consequence of (\ref{eq:accronrange}), it is verified that $DB$ and $B^*D$
   are closed and desely defined adjoint unbounded operators, and the
   topological splitting 
$$
   L_2(\R^n;\C^m)= \nul(D)\oplus B^* \clos{\ran(D)}
$$
 follows from perturbing the orthogonal splitting $L_2=\nul(D)\oplus \clos{\ran(D)}$
 with $B^*$, which satisfies (\ref{eq:accronrange}) as well.
 Since $\clos{\ran(DB)} = \nul(B^*D)^\perp= \nul(D)^\perp= \clos{\ran(D)}$ and
 $\nul(DB)=\ran(B^*D)^\perp=(B^*\ran(D))^\perp$, the splitting (ii) for $DB$
 follows, by stability of splittings under taking orthogonal complements.

 Since $DB=(DB)|_{\clos{\ran(D)}}\oplus 0$ in the splitting (ii),
 it suffices to prove resolvent bounds for $DB$ on $\clos{\ran(D)}$.
 To this end, let $u\in \clos{\ran(D)}\cap \dom(DB)$ and $f= (\lambda I- DB)u$.
 Then
$$
  \im (Bu, \lambda u- DBu) =\im (Bu, f).
$$
Since $D$ is self-adjoint, $(Bu, DBu)\in\R$ and we get an estimate $| \im(\conj\lambda(Bu, u)) |\lesssim \|u\| \|f\|$,
from which the resolvent bound $\|u\| \lesssim \|f\| / \dist(\lambda, S_\omega)$ follows.
\end{proof}

These properties of closed operators of the form $DB$ have been known for
some time, see for example \cite{ADMc} and \cite{CDMcY}, at least in the case
when $B$ is strongly accretive.
The following theorem has also been known for some time in the case when $D$ is injective,  
as it derives from the special case $D= -i d/dx$ developed in \cite{McQ}
 (see also Lecture 8 of \cite{ADMc}). In this case $DB$ is similar to the
operator $\frac{d}{dz}|_\gamma$ of differentiation on a Lipschitz graph
$\gamma$, and the boundedness of $\sgn(DB)$ is equivalent to the
boundedness of the Cauchy singular integral on $\gamma$, proved
originally by Calder\'on when $B-I$ is sufficiently small \cite{Ca}, and in
general by Coifman, McIntosh and Meyer~\cite{CMcM}.

The proof of the following theorem however is more involved when $D$ is not injective.
In the general case it was proved in \cite[Theorem 3.1(iii)]{AKMc},
building on results for the Kato problem by
Auscher, Hofmann, Lacey, McIntosh and Tchamitchian~\cite{AHLMcT}.
It is also possible to give a direct proof, as shown in \cite{elAAM}.
\begin{thm}    \label{thm:qestsfordb}
Let $D$ be a self-adjoint homogeneous first order differential operator with constant coefficients
such that 
$$
   \|Df\|\gtrsim \|\nabla f\|\qquad \text{for all }  f\in \clos{\ran(D)}\cap\dom(D),
$$
and let $B$ be a bounded multiplication operator in $L_2(\R^n; \C^N)$
which satisfies (\ref{eq:accronrange}).
\begin{itemize}
\item[{\rm (i)}]
The operator $DB$ satisfies quadratic estimates
$$
  \int_0^\infty \| tDB (1+ (tDB)^2)^{-1} f \|^2\frac{dt}t \approx \|f\|^2, 
  \qquad \text{for all } f\in\clos{\ran(D)}.
$$
\item[{\rm (ii)}]
The operator $DB$ has a bounded holomorphic functional calculus in $\clos{\ran(D)}$, i.e.~for each bounded holomorphic function $b(z)$ on a double sector 
$S_\nu^o$, $\omega<\nu<\pi/2$, the operator $b(DB)$ in $\clos{\ran(D)}$ is bounded with estimates
$$
   \|b(DB)\|_{ L_2 \rightarrow L_2} \lesssim \| b\|_{L_\infty(S_\nu^0)}.
$$
\end{itemize}
\end{thm}
For the precise definition of operators $b(DB)$ in the functional calculus of $DB$ we refer to
\cite{ADMc}.
Note that the map $H_\infty(S_\nu^o)\ni b \mapsto b(DB)\in \mL(\clos{\ran(D)})$
is a continuous algebra homomorphism.

We now return to the operator $T_A$ of Definition~\ref{defn:hattransforms}.
Note that the isomorphism $\oA$ in (\ref{eq:similarity}) 
maps the subspace $\clos{\ran(D)}$ onto itself. 
\begin{defn}    \label{defn:Hspace}
  Let $\mH$ denote the closed subspace
$$
 \mH:= \clos{\ran(D)}=\nul(\curl_\ta)=\sett{f\in L_2(\R^n;\C^{(1+n)m})}{\curl_x f_\ta=0}.
$$
of $L_2(\R^n;\C^{(1+n)m})$
\end{defn}
In this notation, Proposition~\ref{prop:typeomega} and Theorem~\ref{thm:qestsfordb} with $B=\hat A$, thus have 
the following corollary.
\begin{cor}    \label{cor:fcalcta}
  The operator $T_A$ from Definition~\ref{defn:hattransforms} is an $\omega$-bisectorial 
  operator in $L_2(\R^n;\C^{(1+n)m})$, where $\omega$ is the angle of accretivity of $\hat A$.
Furthermore we have a splitting 
$$
L_2(\R^n;\C^{(1+n)m}) =\clos{\ran(T_A)}\oplus \nul(T_A) = \mH\oplus \sett{[0, f_\ta]^t}{\divv_x A_{\ta\ta}f_\ta=0}
$$
in which $T_A=T_A|_{\mH} \oplus 0$.
The restriction of $T_A$ to $\mH$ is an injective operator with dense range in $\mH$, 
which satisfies quadratic estimates
and has a bounded holomorphic functional calculus.
\end{cor}
Note that the restriction of $T_A$ to $\mH$, which we continue to denote by $T_A$, coincides
with the operator $T_A$ used in \cite{AAH} for $m=1$.

Of importance to us are the following operators, which are bounded
operators in $\mH$ because they belong to the functional calculus of $T_A$.
\begin{itemize}
\item
The characteristic functions
$$
\chi^\pm(z)= \left\{ \begin{array}{ll}
1    & \quad \textrm{if $\pm \re z>0$}\\
0    & \quad \textrm{if $\pm \re z< 0$} \\
\end{array} \right.
$$
which give the generalised {\em Hardy projections} $E^\pm_A := \chi^\pm(T_A)$.
\item
The signum function
$\sgn(z)= \chi^+(z) -\chi^-(z)$
which gives the generalised {\em Cauchy operator} $E_A := \sgn(T_A)= E_A^+-E_A^-$.
\item
The exponential functions $e^{-t|z|}$, $t>0$, which give the operators $e^{-t|T_A|}$.
Note that $|z|:= z\sgn(z)$ does not denote absolute value for non real $z$, but $z\mapsto|z|$ 
is holomorphic on $S^o_{\pi/2}$.
\end{itemize}
Note that the quadratic estimates in Corollary~\ref{cor:fcalcta} can be written
as 
\begin{equation}   \label{eq:genqest}
   \int_0^\infty \|\psi(tT_A)f\|^2 \frac {dt}t\approx \|f\|^2,\qquad f\in \clos{\ran(D)},
\end{equation}
where $\psi(z):= z/(1+z^2)$.
The estimate $\lesssim$ remains valid for any holomorphic $\psi(z)$ on $S^o_\nu$ such that
$|\psi(z)|\lesssim \min(|z|^\alpha, |z|^{-\alpha})$ for some $\alpha>0$.
If furthermore $\psi|_{S^o_{\nu+}}\ne 0$ and $\psi|_{S^o_{\nu-}}\ne 0$, then 
the estimate $\approx$ holds.
See \cite{ADMc}.

\begin{proof}[Proof of Theorem~\ref{thm:mainhardy}]
By Corollary~\ref{cor:fcalcta}, the infinitesimal generator has a bounded holomorphic
functional calculus.
With the notation introduced above, 
define Hardy type subspaces $E_A^\pm\mH:= \sett{E_A^\pm f}{f\in\mH}$,
so that 
$$
\mH = E_A^+\mH \oplus E_A^-\mH.
$$
In terms of the operator $T_A$, the vector fields $f$ and $F^\pm$ are related as
$$
  f=f^++f^- \longleftrightarrow E^\pm_A f= f^\pm = F^\pm|_{\R^n} 
  \longleftrightarrow e^{\mp t|T_A|} f^\pm = F^\pm.
$$
Indeed, differentiating $F^\pm= e^{\mp t|T_A|}f^\pm$ at $(t,x)\in\R^{1+n}_\pm$, we have
$$
  \pd_t F^\pm= \mp |T_A| F^\pm = -T_A e^{\mp t|T_A|}(\pm E_Af^\pm)
  =-T_A e^{\mp t|T_A|}f^\pm = -T_A F^\pm
$$
since $f^\pm \in E_A^\pm\mH$. Thus $F^\pm$ satisfies (\ref{eq:generator}), 
which we have seen at the beginning of this section is equivalent to
(\ref{eq:Laplacein1order}).
Conversely, each vector field
$F \in C^1(\R_+; L_2(\R^n;\C^{(1+n)m}))$
which satisfies (\ref{eq:generator}) for $t>0$ and has
limits $\lim_{t\rightarrow \infty}F_t=0$ and $\lim_{t\rightarrow 0}F_t=f$
in $L_2$ norm, is of the form $F(t,x)= e^{-t|T_A|}f$ and $f\in E_A^+\mH$.
To see this, split $F_t= F_t^+ + F_t^-$ where $F_t^\pm\in E_A^\pm \mH$.
Applying $E^\pm_A$ to (\ref{eq:generator}) gives the equations
$\pd_t F_t^\pm \pm |T_A| F_t^\pm=0$. Multiplying with suitable exponentials shows that,
for fixed $t>0$,
$e^{(s-t)|T_A|}F_s^+$ is constant for $s\in(0,t)$ and
$e^{(t-s)|T_A|}F_s^-$ is constant for $s\in(t,\infty)$.
Thus $F_t^-=0$ and $F_t^+= e^{-t|T_A|}f$, where $f= \lim_{t\rightarrow 0} F_t$. 
The corresponding result for $F^-$ is proved similarly.

The equivalences of norms follow from the quadratic estimates for $T_A$ as follows.
Boundedness of the complementary projections $E_A^\pm$ shows that
$\|f\|\approx \|f^+\|+\|f^-\|$,
the uniform boundedness of $e^{- t |T_A|}$, $t>0$, shows that 
$\|f^\pm\| \approx \sup_{\pm t>0} \|F^\pm_t\|$, and
the quadratic estimates for $T_A$ shows that
$\tb{t\pd_t F^\pm_{\pm t}}\approx \|f^\pm\|$.
Finally $ \|f^\pm\|\approx \|\widetilde N_* (F^\pm)\|$ 
was proved in \cite[Proposition 2.56]{AAH} for $m=1$.
The extension to divergence form systems is straightforward.

To verify Lipschitz continuity, one shows that the operators
\begin{align}
   f\longmapsto b(T_A)f &: \mH\longrightarrow \mH,   \label{eq:anal1} \\
   f\longmapsto ( b(tT_A)f)_{t>0} &: \mH\longrightarrow L_2((a,b); \mH),  \label{eq:anal2} \\
   f\longmapsto (\psi(tT_A)f)_{t>0} &: \mH\longrightarrow L_2(\R_+, dt/t; \mH)  \label{eq:anal3}
\end{align}
depend analytically on $A$, where $b$ and $\psi$ are bounded holomorphic function on $S^0_\nu$,
$\psi$ decays at $0$ and $\infty$ as in (\ref{eq:genqest}), and the interval $(a,b)$ is finite.
One proceeds similar to \cite[Theorem 6.4]{AKMc} and  \cite[Lemma 2.41]{AAH},
starting from the analyticity of resolvents $A\mapsto (\lambda I- T_A)^{-1}$,
using the quadratic estimates from Corollary~\ref{cor:fcalcta} for the operator $T_A$
and the fact the uniform limits of analytic functions are analytic.

Lipschitz continuity can now be deduced from analyticity. Given $A_1$ and $A_2\in L_\infty(\R^n; \C^{(1+n)m})$
which are strictly accretive on $\mH$, define 
$$
  A(\zeta):= A_1+ \zeta(A_2-A_1)/\|A_2-A_1\|_\infty
$$ 
so that $\zeta\mapsto A(\zeta)$ is analytic in a neighbourhood $\Omega$
of the interval $[0, \|A_2-A_1\|_\infty]$.
Consider the analytic function $\zeta\mapsto A(\zeta)\mapsto F^\pm_t= b(tT_{A(\zeta)})f$, 
where $b(z):= e^{-t|z|}\chi^\pm(z)$,
which has bounds $\|F_t^\pm\|_2\lesssim \|f\|_2$ in $\Omega$.
Thus $\|dF_t^\pm/d\zeta\|_2\lesssim \|f\|_2$ in $\Omega$, from which Lipschitz continuity 
of $A\mapsto F_t$ follows, uniformly for all $t>0$.

Lipschitz continuity of $A\mapsto F_t$ for the square function norm and the norm of the 
non-tangential maximal function are proved similarly,
using analyticity of (\ref{eq:anal3}), with $\psi(z)=ze^{-|z|}\chi^\pm(z)$, and (\ref{eq:anal2}), 
with $b(z)= e^{-t|z|}\chi^\pm(z)$, respectively.
For the non-tangential maximal function, we refer to the proof of \cite[Theorem 1.1]{AAH}
for further details.
\end{proof}

\section{Dirichlet and Neumann boundary value problems}    \label{sec:divbvp}

In this section, we return to the Dirichlet and Neumann BVP's and use 
Theorem~\ref{thm:mainhardy} to prove Theorem~\ref{thm:bvpfordivform}.
We start by translating (Dir-$A$) to an auxiliary Neumann problem (Neu$^\perp$-$A$),
which consists in finding $U$ solving (\ref {eq:divform}) with boundary condition
\begin{itemize}
\item
Neumann problem (Neu$^\perp$-$A$): $-\pd_t U(0,x)=\varphi(x)$, 
where $\varphi\in L_2(\R^n;\C^m)$ is given.
\end{itemize}
More precisely, we use the following first order formulation of well posedness.
\begin{defn}
    We say that the boundary value problem (Neu$^\perp$-$A$)
  is {\em well posed} if for each boundary data $\varphi\in L_2(\R^n;\C^m)$, 
  there exists a unique vector field
$$
  F(t,x)= F_t(x)= \nabla_{t,x}U(t,x) \in C^1(\R_+; L_2(\R^n;\C^{(1+n)m}))
$$
which satisfies (\ref{eq:Laplacein1order}) for $t>0$, and has
limits $\lim_{t\rightarrow \infty}F_t=0$ and $\lim_{t\rightarrow 0}F_t=f$
in $L_2$ norm, where the full boundary trace $f$ satisfies the boundary condition $-f_0=\varphi$.
\end{defn}
\begin{lem}     \label{lem:conjfcnreduction}
  Given $u=-\varphi\in L_2(\R^n;\C^m)$, we have a one-to-one correspondence 
$$
  U(t,x)= F_0(t,x) \longleftrightarrow F(t,x)=  -\int_{t}^\infty \nabla_{s,x}U(s,x)\, ds
$$
between solutions $U(t,x)$ to (Dir-$A$) and solutions $F(t,x)$ to (Neu$^\perp$-$A$).
In particular WP(Dir)= WP(Neu$^\perp$), where
$$
  \text{WP(Neu$^\perp$)} := \sett{A}{\text{(Neu$^\perp$-$A$) is well posed}}
  \subset L_\infty(\R^n; \C^{(1+n)m}).
$$
\end{lem}
\begin{proof}
  Assume that $F$ solves (Neu$^\perp$-$A$) with boundary condition $f_0= u$,
  and let $U:= F_0$.
  Then $U_t\in C^1(\R_+;L_2)$ and $\nabla_{x} U = \pd_t F_\ta\in C^0(\R_+;L_2)$.
  The limits $\lim_{t\rightarrow 0}U_t=u$,
  $\lim_{t\rightarrow \infty}U_t=0$ and $\lim_{t\rightarrow \infty}\nabla_{t,x} U_t=0$ 
  are direct consequences of the limits 
  $\lim_{t\rightarrow 0}F_t=f$, $\lim_{t\rightarrow\infty} F_t= \lim_{t\rightarrow\infty}t\pd_t F_t=0$,
  whereas
$$
  \int_{t_0}^{t_1} \nabla_{x} U_s\, ds= \int_{t_0}^{t_1} \pd_s F_\ta\, ds= 
  F_\ta(t_1)-F_\ta(t_0)\longrightarrow -f_\ta,\qquad (t_0,t_1)\longrightarrow (0,\infty).
$$
  The function $U$ satisfies (\ref {eq:divform}) since
\begin{multline*}
  \int_t^\infty (\nabla_x v, (A\nabla_{s,x}U_s)_\ta)\,ds=
  \int_t^\infty \pd_s(\nabla_x v, (A F_s)_\ta) \, ds  \\ = -(\nabla_x v, (AF_t)_\ta)
  =(v, \divv_x (AF_t)_\ta)= -(v, \pd_t(A F_t)_0)= -(v,(A\nabla_{t,x}U_t)_0).
  \end{multline*}

Conversely, assume that $U$ solves (Dir-$A$) with boundary condition $U|_{\R^n}= -\varphi$,
and let $F(t,x):=  -\int_{t}^\infty \nabla_{s,x}U(s,x)\, ds$.
This gives a well defined function since $F= U-\lim_{t_1\rightarrow \infty}\int_{t}^{t_1} \nabla_x U_s\, ds$,
and $F\in C^1(\R_+; L_2)$.
Clearly, $\lim_{t \rightarrow \infty} F_t=0$ and $\lim_{t\rightarrow 0}F_t= -\varphi -h$,
where $h:= \int_{0}^{\infty} \nabla_x U_s\, ds$.
The vector field $F$ satisfies (\ref{eq:Laplacein1order}) since
$\curl_{x}\int_{t}^\infty \nabla_{x}U(s,x)\, ds=0$, $\pd_t F_\ta=\nabla_x U= \nabla_x F_0$ and
\begin{multline*}
  (v, \divv_x (AF_t)_\ta)=  (\nabla_x v, \int_{t}^\infty (A\nabla_{s,x}U_s)_\ta \, ds) \\
  = \int_t^{\infty} (\nabla_x v, (A \nabla_{s,x}U_s)_\ta)\, ds= -(v, (A \nabla_{t,x}U_t)_0)
  = -(v, \pd_t (AF_t)_0),
\end{multline*}
for all $v\in C^\infty_0(\R^n;\C^m)$, so that $\pd_t(AF_t)_0=-\divv_x(AF_t)_\ta$.
This completes the proof.
\end{proof}

As in the proof of Theorem~\ref{thm:mainhardy}, we denote the upper Hardy type subspace of
$\mH$ by
$$
  E_A^+ \mH = \sett{f}{F\in C^1(\R^{1+n}_+; \mH) \text{ solves } (\ref{eq:generator}) \text{ in } \R^{1+n}_+, \,
  \lim_{t\rightarrow 0^+}\|F_t-f\|_2= \lim_{t\rightarrow \infty}\|F_t\|_2=0}.
$$

\begin{proof}[Proof of Theorem~\ref{thm:bvpfordivform}]
By Theorem~\ref{thm:mainhardy} , the spectral projection $E_A^+= \chi_+(T_A)$ 
onto this subspace is bounded and depends
Lipschitz continuously on $A$.
We now observe that (Reg-$A$), (Neu-$A$)  and (Neu$^\perp$-$A$) are well posed if and only if
\begin{align*}
  E_A^+\mH \longrightarrow \sett{g\in L_2(\R^n;\C^{nm})}{\curl_x g=0} &: f\longmapsto f_\ta, \\
  E_A^+\mH \longrightarrow L_2(\R^n;\C^m) &: f\longmapsto (Af)_0, \\
  E_A^+\mH \longrightarrow L_2(\R^n;\C^m) &: f\longmapsto f_0
\end{align*}
are isomorphisms respectively.
Since $E_A^+$ depends continuously on $A$ by Theorem~\ref{thm:mainhardy}, the following lemma shows
that the sets of well posedness are open. 
\begin{lem}   \label{lem:vardomains}
Let $P_t$ be bounded projections in a Hilbert space $\mH$ which depend continuously on a parameter $t\in(-\delta,\delta)$,
and let $S:\mH\rightarrow\mK$ be a bounded operator into a Hilbert space $\mK$.
If $S: P_0\mH\rightarrow \mK$ is an isomorphism, then there exists $0<\epsilon<\delta$,
such that $S:P_t \mH\rightarrow\mK$ is an isomorphism when $|t|<\epsilon$.
\end{lem}
\begin{proof}
Consider the family of operators $P_0\mH\ni f \mapsto S P_t f\in \mK$ between fixed spaces.
By assumption and continuous dependence, they are invertible when $|t|$ is small. 
Thus it suffices to prove that $P_t: P_0 \mH\rightarrow P_t\mH$ is invertible when $|t|$ is small.
This holds since $(I-P_0(P_0-P_t))^{-1} P_0$, $P_0(I-P_t(P_t-P_0))^{-1}:  P_t \mH\rightarrow P_0\mH$ 
are seen to be left and right inverses respectively.
\end{proof}
What remains to be proved is that the three maps are isomorphisms when $A$ is either
Hermitean, block or constant.
In fact, it suffices to prove this for (Reg-$A$) and (Neu-$A$), due to the following result proved in
\cite[Proposition~2.52]{AAH}.
\begin{prop}   \label{prop:regneuduality}
  The boundary value problem (Neu$^\perp$-$A$) is well posed if and only if (Reg-$A^*$) is well posed.
\end{prop}
That (Reg-$A$) and (Neu-$A$) are well posed for Hermitean, block and constant coefficents, follows
from \cite[Section 3]{AAH} when $m=1$. 
For Hermitean and block form coefficients, these proofs are readily adapted to systems, but to
be self contained, we give simplified proofs below.

Define the operator $N:= \begin{bmatrix} -1 & 0 \\ 0 & I\end{bmatrix}$, which reflects
a vector in $\R^n$. 
Write $N^+=\tfrac 12(1+N)=\begin{bmatrix} 0 & 0 \\ 0 & I\end{bmatrix}$ for the tangential projection and 
$N^-=\tfrac 12(1-N)=\begin{bmatrix} 1 & 0 \\ 0 & 0\end{bmatrix}$ for the normal projection,
so that $N=N^+-N^-$.

\subsection{Hermitean matrices}   \label{sec:hermitean}

Let $f\in E^+_A\mH$.
This means that there is a vector field $F_t$ in $\R^{1+n}_+$ such that $\pd_t F_t=-T_A F_t$,
$\lim_{t\rightarrow\infty}F_t=0$ and $\lim_{t\rightarrow 0}F_t=f$.
Recall that
$$
   T_A= \oA^{-1}  D \uA
$$
and note that $D N +N D=0$.
Furthermore, assuming that $A^*=A$, it is seen from the definition of $\hat A$ that the Hermitean
condition translates to $(\hat A)^*=N\hat A N$.
The Rellich type identity which is useful here is the following.
\begin{multline*}
   (N\uA f, \oA f)=-\int_0^\infty \pd_t (N\uA F_t, \oA F_t) dt
  = \int_0^\infty (N\uA T_A F_t, \oA F_t)+ (N\uA F_t, \oA T_A F_t) dt \\
  = \int_0^\infty (N\hat A D \uA F_t, \oA F_t)+ (N\uA F_t, D \uA F_t) dt
  = \int_0^\infty ((N D+ D N)\uA F_t, \uA F_t) dt=0
\end{multline*}
Thus $((\uA f)_0, (\oA f)_0)= ((\uA f)_\ta,(\oA f)_\ta)$, or in view of (\ref{eq:uoa}),
\begin{equation}   \label{eq:rellich}
  ((f)_0, (A f)_0)= ((A f)_\ta,(f)_\ta).
\end{equation}
Consider first the Neumann problem. From (\ref{eq:rellich}) it follows that
$$
  \|f\|^2\approx \re(Af,f)= \re\big((Af)_0, f_0)+ ((Af)_\ta, f_\ta)\big)
  = 2\re((Af)_0, f_0)\lesssim \|(Af)_0\| \|f\|.
$$
This shows that $\|f\|\lesssim \|(Af)_0\|$ holds for the Neumann map
$E_A^+\mH\ni f\mapsto (Af)_0$, which implies that this map is injective
with closed range. 

It remains for us to prove surjectivity of this map. Note that the above estimates also show that
$E_{A_t}^+\mH\ni f\mapsto (A_tf)_0$  is injective with closed range when 
$A_t:= (1-t)I+ tA$, $0\le t\le 1$.
It follows from the proof of Lemma~\ref{lem:vardomains} and the method of 
continuity that all these maps have the same index. 
Since $I=A_0\in WP$, it follows that $A\in WP$.

Well posedness for (Reg-$A$) and (Neu$^\perp$-$A$) is proved in a similar way,
by keeping the factor $f_\ta$ and $f_0$ respectively from (\ref{eq:rellich}).

\subsection{Block matrices}   \label{sec:block}

Note that in this case the Neumann problems (Neu-$A$) and (Neu$^\perp$-$A$)
coincide, and that the operator $T_A$ has the form 
$$
  T_A= \begin{bmatrix}
    0 & A_{00}^{-1} \divv_x A_{\ta\ta} \\
    -\nabla_x & 0
  \end{bmatrix}.
$$
Note that in this case, the accretivity condition (\ref{eq:accrassumption}) splits
into the two independent assumptions $\re(A_{00} u,u)\gtrsim\|u\|_2^2$
and $\re(A_{\ta\ta}\nabla_x v,\nabla_x v)\gtrsim\|\nabla_x v\|_2^2$ 
for all $u\in L_2(\R^n;\C^m)$ and $\nabla_x v\in L_2(\R^n; \C^{nm})$.
Since the diagonal elements in $T_A$ are zero, so are the diagonal elements of 
$E_A$ since
\begin{equation}        \label{blockmatrix}
  E_A= T_A (T_A^2)^{-1/2} = (T_A^2)^{-1/2}T_A = 
  \begin{bmatrix}
    0 & L^{-1/2} A_{00}\divv_x A_{\ta\ta} \\
    -\nabla_x L^{-1/2} & 0
  \end{bmatrix},
\end{equation}
where $L:= -A_{00}\divv_x A_{\ta\ta}\nabla_x$.
Another way to see this is from the calculation
$$
  N E_A= N\sgn(T_A)N^{-1}N= \sgn(NT_AN^{-1})N= \sgn(-T_A)N= -E_AN.
$$
To prove well posedness, we need to prove that 
$N^+: E^+_A \mH \rightarrow N^+\mH$ and
$N^-: E^+_A \mH \rightarrow N^-\mH$
are isomorphisms. 
From $E_AN+NE_A=0$, we obtain explicit inverses as
$$
  2E_A^+: N^+\mH \longrightarrow E_A^+\mH \quad \text{and} \quad
  2E_A^+: N^- \mH \longrightarrow E_A^+\mH.
$$
For example, to see that $N^+(2E_A^+g)=g$ when $g\in N^+\mH$, we calculate
\begin{multline*}
  2N^+E_A^+g = N^+(I+E_A)g= g+ N^+E_Ag \\
  = g+ \tfrac 12 (1+N)E_Ag= 
  g+\tfrac 12(E_Ag- E_ANg)=g.
\end{multline*}

In fact well posedness of the Neumann and regularity problems for block coefficients
is equivalent to the Kato square root estimate
$$
  \|\sqrt{L} u\|_2 \approx \|\nabla_x u\|_2,
$$
as was first observed by Kenig~\cite[Remark 2.5.6]{kenig}.
To see this, we deduce from Equation (\ref{blockmatrix}) that $f\in E_A^+\mH$, i.e. $f=E_A f$, if and only if
$f_\ta=-\nabla_x L^{-1/2} f_0$, or inversely $f_0= L^{-1/2} A_{00}\divv_x A_{\ta\ta} f_\ta$ and
$\curl_x f_\ta=0$.
This can be used to construct $f= f_0+ E_A f_0= E_A f_\ta+ f_\ta$ from either $f_0$ or $f_\ta$.
Note that the Kato estimate translates to $\|f_0\|\approx \| f_\ta\|$.

\subsection{Constant matrices}

If $A$ is a constant matrix, we can make use of the Fourier transform.

(i)
First consider the simpler case when $m=1$.
In this case, the solutions to the BVPs can be explicitly computed on the Fourier
transform side, since the problem reduces to an eigenvector calculation.
At the frequency point $\xi\not = 0$, the space $\mH$ corresponds to the
two dimensional space $\mH_\xi:= \sett{ze_0+ w\xi}{z, w\in \C}$ of
vectors with tangential part parallel to $\xi$.
The operator $D$ corresponds to
$
  D_\xi:= 
  \begin{bmatrix}
    0 & i\xi^t \\
    -i\xi & 0
  \end{bmatrix}.
$
Compressing the constant matrix $\hat A$ to $\mH_\xi$, we define the 
$2\times 2$ strictly accretive matrix
$
  \begin{bmatrix}
    a & b \\
    c & d
  \end{bmatrix}
  :=
   \begin{bmatrix}
    \hat A_{00} & \hat A_{0\ta}\xi \\
    \xi^t \hat A_{\ta 0} & \xi^t \hat A_{\ta\ta}\xi
  \end{bmatrix}
$,
so that 
$$
  D_\xi \hat A
  \sim
    i\begin{bmatrix}
    c & d \\
    -a & -b
  \end{bmatrix}.
$$
Computing eigenvalues and vectors shows that $ze_0+ w\xi\in\chi_\pm(D_\xi \hat A)$
if and only if 
\begin{equation}   \label{eq:constrellich}
  (az+bw)= \Big( \tfrac 12(c-b)\pm i\sqrt{ad-\tfrac 14(b+c)^2}\Big)w.
\end{equation}
Applying the similarities in (\ref{eq:similarity}), we characterize well posedness
as follows.
That (Neu-$A$), (Reg-$A$) and (Neu$^\perp$-$A$) are well posed means that $ze_0+ w\xi\in\chi_\pm(D_\xi \hat A)$ is 
determined by $z$, $w$ and $az+bw$ respectively.
This is straightforward to verify using (\ref{eq:constrellich}).

(ii)
Next consider the case $m>1$.
In this case, we perform a Rellich type argument on the Fourier symbol, or
rather we make a ``reverse Rellich estimate''. 

The space $\mH_\xi$ is now isomorphic to $\C^{2m}$ since $z,w\in\C^m$.
In view of Proposition~\ref{prop:regneuduality}, it suffices
to prove a-priori estimates  $\|f\|\lesssim\|f_\ta\|$ and
$\|f\|\lesssim\|(Af)_0\|$ 
for $f\in\chi_+(T_{\xi})\mH^k_\xi$ uniformly for almost all $\xi\in\R^n$, 
where $T_\xi:= \oA^{-1}D_\xi\uA$.
However, since $\chi_+(tT_{\xi})= \chi_+(T_{\xi})$
for $t>0$, it suffices to consider the unit sphere $|\xi|=1$.
By continuity and compactness, we need only to verify that no non zero vector $f$
such that $f_\ta=0$ or $(Af)_0=0$ can be in the Hardy space, i.e.~be of the form $f=F(0)$,
where $F:\R_+\rightarrow\mH^k_\xi$ satisfies 
$\pd_t F= -T_{\xi}F$ and $\lim_{t\rightarrow\infty}F=0$.
To prove this, we use the fact that $D_\xi^2=I$ to obtain
\begin{multline*}
   ( D_\xi \oA f, \oA f )
   = -\int_0^\infty \pd_t  ( D_\xi \oA F(t), \oA F(t) )\,dt \\
   = 2\int_0^\infty \re(\uA F(t), \oA F(t)) \,dt= 2\int_0^\infty \re(A F(t), F(t)) \,dt.
\end{multline*}
We now observe that the left hand side vanishes if $f_\ta=0$ or $(Af)_0=0$,
and from the right hand side we then see that $F=0$ identically, and therefore $f=0$.
The method of continuity, perturbing $A$ to $I$ now shows that the maps 
$f\mapsto f_\ta$ and $f\mapsto (Af)_0$ are surjective, and thus isomorphisms.

We have now completed the
proof of Theorem~\ref{thm:bvpfordivform}.
\end{proof}

\begin{proof}[Proof of Corollary~\ref{cor:DJK}]
As in Lemma~\ref{lem:conjfcnreduction} a function $U$ solving (\ref{eq:divform}),
with properties as in Definition~\ref{defn:wellposed}, is the normal part of a vector field
$F= U+ F_\ta$ solving (\ref{eq:Laplacein1order}), with properties as in 
Definition~\ref{defn:wellposed}.
Theorem~\ref{thm:mainhardy} shows that $F= e^{-t|T_A|}f$, where $f= F|_{\R^n}\in E_A^+\mH$
and that estimates
$$
  \|f\|_2 \approx \sup_{t>0} \|F_t\|_2\approx \|\widetilde N_* (F)\|_2  \approx \tb{t\pd_t F_t} 
$$
hold.
If $A\in WP(Dir)= WP(Neu^\perp)$, then $\|f\|_2\approx \|u\|_2$ and
$\|F_t\|_2\approx \|U_t\|_2$ for all $t>0$, since $F_t\in E_A^+\mH$.
For the square function norm we observe that $\pd_t F= \nabla_{t,x}U$,
and for the non-tangential maximal function clearly 
$\|\widetilde N_* (F)\|_2\gtrsim \|\widetilde N_* (U)\|_2$ holds.
As $\|u\|_2\approx \sup_{t>0}\|U_t\|_2$ for solutions to (\ref{eq:divform}) has been shown, 
we have $\|U_t\|_2\lesssim \|U_s\|_2$
when $t>s$. The reverse estimate $\|\widetilde N_* (F)\|_2\lesssim \|\widetilde N_* (U)\|_2$ 
now follows from
\begin{multline*}
  \|\widetilde N_*(U)\|^2\gtrsim \sup_{t>0} \int_{\R^n}
\barint_{\hspace{-6pt} |y-x|<c_1t}\barint_{\hspace{-6pt} |s-t|<c_0t} |U(s,y)|^2 \, dsdydx \\
= \sup_{t>0} \barint_{\hspace{-6pt} |s-t|<c_0t} \|U_s\|^2 ds \gtrsim 
\sup_{t>0} \|U_{(1+c_0)t}\|^2\approx \|u\|^2.
\end{multline*}
This proves the corollary.
\end{proof}

\section{Uniqueness of solutions}   \label{sec:unique}

In this section we compare the solutions to the BVP's (Neu-$A$), (Dir-$A$) and (Reg-$A$)
in the sense of Definition~\ref{defn:wellposed}, with the standard solutions obtained 
from the Lax--Milgram Theorem.
This uses the homogeneous Sobolev space $\dot H^1(\R^{1+n}_+;\C^m)$, equipped with the norm
$\|U\|^2_{\dot H^1}:= \int_{\R^{1+n}_+}|\nabla_{t,x} U|^2$, and the subspace of functions with 
vanishing trace.
Continuing our first order approach to BVP's via (\ref{eq:Laplacein1order}), 
we make the following definition.
\begin{defn}
Introduce spaces of vector fields
\begin{align*}
  L_2^\nabla(\R^{1+n}_+;\C^{(1+n)m}) &:= \sett{F\in L_2(\R^{1+n}_+;\C^{(1+n)m})}{\curl_{\R^{1+n}_+}F=0} \quad \text{and} \\
  L_2^{\nabla_0}(\R^{1+n}_+;\C^{(1+n)m}) &:= \sett{F \in L_2(\R^{1+n}_+;\C^{(1+n)m})}{\curl_{\R^{1+n}} (F_z)=0}.
\end{align*} 
The condition $\curl_{\R^{1+n}} (F_z)=0$ here means that the extension by zero $F_z$, 
of $F$ to $\R^{1+n}$, is curl free, 
or formally: $\curl_{\R^{1+n}_+}F=0$ and the boundary trace of $F$ is normal to $\R^n$.
If $F\in L_2^\nabla(\R^{1+n}_+;\C^{(1+n)m})$, it is seen that there there exists $U\in L_2^{loc}(\R^{1+n}_+;\C^m)$, unique 
up to constants among the distributions on $\R^{1+n}_+$, such that $\nabla_{t,x} U= F$.
Define Hilbert spaces
\begin{align*}
  \dot H^1(\R^{1+n}_+;\C^m) &:= \sett{U\in L_2^{loc}(\R^{1+n}_+;\C^m)} {\nabla_{t,x}U\in L_2^{\nabla}(\R^{1+n}_+;\C^{(1+n)m})}, \\
  \dot H^1_0(\R^{1+n}_+;\C^m) &:= \sett{U\in L_2^{loc}(\R^{1+n}_+;\C^m)} {\nabla_{t,x}U\in L_2^{\nabla_0}(\R^{1+n}_+;\C^{(1+n)m})},
\end{align*}
with norms so that the correpondence $U\leftrightarrow F=\nabla_{t,x}U$ is an isometry.
\end{defn}
It is straightforward to verify that a function $U\in \dot H^1(\R^{1+n}_+;\C^m)$ belongs to the subspace 
$\dot H^1_0(\R^{1+n}_+;\C^m)$ if and only if there exists a constant $C$ such that $U$ extended by
$C$ to $\R^{1+n}$ belongs to $\dot H^1(\R^{1+n};\C^m)$.

Functions $U\in \dot H^1(\R^{1+n}_+;\C^m)$ are well defined only up to constants, 
whereas for $U\in \dot H^1_0(\R^{1+n}_+;\C^m)$, we 
will choose the constant so that $U|_{\R^n}=0$.
It is not true that $\dot H^1_0(\R^{1+n}_+;\C^m)\subset L_2(\R^{1+n}_+;\C^m)$, as a scaling argument 
readily shows. However, Poincar\'e's inequality shows that
$$
  \iint_{\R^{1+n}_+} |U(t,x)|^2\frac{dtdx}{1+t^2+|x|^2} \lesssim \iint_{\R^{1+n}_+} |\nabla_{t,x}U(t,x)|^2
  \,dtdx,\qquad U\in \dot H^1_0(\R^{1+n}_+;\C^m).
$$

If $F=\nabla_{t,x} U$ solves (\ref{eq:Laplacein1order}), then we formally have
$$
   J_A(U, V) := \iint_{\R^{1+n}_+} (A \nabla_{t,x}U, \nabla_{t,x} V) \, dtdx = -\int_{\R^n} ((Af)_0, v) \, dx,
$$
where $f= F|_{\R^n}$ and $v=V|_{\R^n}$.
As pointed out in Section~\ref{sec:notres}, the standing assumption that $A$ is strictly accretive
on $\nul(\curl_\ta)$, i.e. (\ref{eq:accrassumption}), implies that the G\aa rding inequality
(\ref {eq:uppergarding}) in $\R^{n+1}_+$ holds, i.e.  $|J_A(U,U)|\gtrsim \|U\|_{\dot H^1}^2$.

Note that $V\mapsto  \int_{\R^n} (\phi(x), v(x))\, dx$ in Equation (\ref{eq:lmneu}) below 
defines a bounded functional on $\dot H^1(\R^{1+n}_+;\C^m)$ if 
$\phi=\divv_x w$, where $\phi\in L_2(\R^n;\C^m)$, $w\in L_2(\R^n;\C^{nm})$,
since $\int |\hat\phi(\xi)|^2 \max(|\xi|^{-1} ,1) d\xi<\infty$ and the trace map $V\mapsto v$
maps $\dot H^1(\R^{1+n}_+;\C^m)\rightarrow \dot H^{1/2}(\R^n;\C^m)$.
Furthermore 
$V\mapsto J_A(P_t u, V)$ defines a bounded functional on $\dot H^1_0(\R^{1+n}_+;\C^m)$ 
if $u\in H^1(\R^n;\C^m)$ since
$\iint  |\xi e^{-t|\xi|} \hat u(\xi)|^2 dtd\xi\approx \|u\|^2_{H^{1}(\R^n)}<\infty$, 
where $P_t$ denotes the Poisson extension
$$
  P_t u(x) := \frac{\Gamma((1+n)/2)}{\pi^{(1+n)/2}}
  \int_{\R^n} \frac{t \, u(y)\, dy}{(t^2+|x-y|^2)^{(1+n)/2}}.
$$
The Lax--Milgram Theorem proves the existence and uniqueness of the following
$\dot H^1$ solutions $U$.
\begin{defn}
  We say that $\phi$ is {\em good boundary data} for (Neu-$A$) if 
  $L_2(\R^n;\C^m)\ni\phi=\divv_x w$, where $w\in L_2(\R^n;\C^{nm})$.
  If $\phi$ is good, we define {\em the $\dot H^1$ solution} to the Neumann problem
  to be the unique function $U\in \dot H^1(\R^{1+n}_+;\C^m)$ such that
\begin{equation}   \label{eq:lmneu}
  J_A(U, V) = 
  \int_{\R^n} (\phi(x), v(x))\, dx, \qquad\text{for all } V\in \dot H^1(\R^{1+n}_+;\C^m).
\end{equation}

We say that $u$ is {\em good boundary data} for (Dir-$A$), or equivalently that
$\nabla_x u$ is {\em good boundary data} for (Reg-$A$),  if 
$u\in H^1(\R^n;\C^m)$.
  If $u$ is good, we define {\em the $\dot H^1$ solution} to the Dirichlet (regularity) problem
  to be the unique function $U\in \dot H^1(\R^{1+n}_+;\C^m)$ such that
\begin{equation}   \label{eq:lmdir}
  J_A(U,  V) = 0, 
  \qquad\text{for all } V\in\dot H^1_0(\R^{1+n}_+;\C^m),
\end{equation}
and $U(t,x)- P_t u(x)\in \dot H^1_0(\R^{1+n}_+;\C^m)$.
\end{defn}
The goal in this section is to prove the following uniqueness result.
\begin{thm}    \label{thm:uniqueness}
  Assume that $A$ belongs to the connected component of WP(Neu) / WP(Reg) / WP(Dir) 
  that contains $I$. 
  If the boundary data is good, then
  the solutions to (Neu-$A$) / (Reg-$A$) / (Dir-$A$) in the sense of 
  Definition~\ref{defn:wellposed}, coincide with the $\dot H^1$ solutions.
\end{thm}
\begin{rem}
  For general $A$ in the set of well posedness, the solutions constructed in this paper using
  the boundary equation method do not necessarily coincide with the $\dot H^1$
  solutions. Examples of this were shown in \cite{AxNon}.
  Note that these examples combined with Theorem~\ref{thm:uniqueness} proves the existence of many 
  coefficients that do not have well posed BVP's (even in when $n=m=1$ with real $A$), 
  sufficiently many to disconnect these $A$ with non-$\dot H^1$ solutions from the identity.
\end{rem}
The proof of Theorem~\ref{thm:uniqueness} uses the following lemma with $A_0=I$.
\begin{lem}   \label{lem:unique}
  Let $A_0$ be a block matrix. 
  Then there exists $\epsilon>0$, such that if $\| A-A_0 \|_\infty <\epsilon$
  and the boundary data is good, then
  the solutions to the BVP's in the sense of 
  Definition~\ref{defn:wellposed}, coincide with the $\dot H^1$ solutions.
\end{lem}
\begin{proof}[Proof for (Neu-$A$)]
   Let $F=\nabla_{t,x} U= e^{-t|T_A|}f$ be the boundary equation solution to (Neu-$A$) in 
   $\R^{1+n}_+$ with data $\phi=\divv_ x w= -(Af)_0$.
   Using the isomorphism $\oA$ from (\ref{eq:similarity}), we define the 
   similar Hardy function $\tilde f:= \oA f\in \chi_+(D\hat A)\mH$.
   
   Let $N^-$ and $N^+$ be the normal and tangential projections
   from Section~\ref{sec:divbvp}.
   The boundary condition on $\tilde f$ can be written
   $N^-\tilde f= -\phi$.
  We solve for $\tilde f$ by making the ansatz $\tilde f= 2\chi_+(D\hat A)h$, where
  $h\in N^-\mH$ (i.e. $h_\ta=0$). This yields the equation
$$
  -\phi= N^-\tilde f = 2N^-\chi_+(D\hat A)h = (I+ N^- \sgn(D\hat A))h
$$
for $h$, in the normal subspace $N^-\mH$.
We note the following properties of  ``the double layer type operator''
$K_A:= N^- \sgn(D\hat A)N^-$.  (See \cite{Axthesis} for explanations of this terminology.)
When $A=A_0$, then as in Section~\ref{sec:block} it follows that $K_{A_0}=0$,
since the diagonal entries of $\sgn(D\hat A)$ are zero in the normal/tangential splitting
of the space.
Theorem~\ref{thm:mainhardy} shows that $K_A$ depends continuously on $A$.
Moreover $K_A(\ran(D)) \subset \ran(D)$
since 
$$
  K_ADg= N^- \sgn(D\hat A)DN^+g=  N^- D\sgn(\hat AD)N^+g=   D\big(N^+\sgn(\hat A D)N^+g\big)
$$ 
when $g\in\dom(D)$.
Therefore, when $\|A - A_0\|_\infty$ is small, we can expand $(I+K_A)^{-1}$ in a Neumann
series and deduce that $h\in\ran(D)$ since $-\phi= -Dw\in\ran(D)$.
Indeed
$$
  \sum_{k=0}^N (-K_A)^k (-Dw)= D\Big(- \sum_{k=0}^N (-N^+ \sgn(\hat A D)N^+)^k w \Big)=: Dw_N,
$$
where $w_N$ and $Dw_N\rightarrow h$ converges in $L_2$. Since $D$ is closed, 
$h\in\ran(D)$.
This shows that $\tilde f= 2\chi_+(D\hat A)h\in \ran(D\hat A)$, and thus $f\in \ran(T_A)$.
In particular $f=|T_A|^{1/2}f_0$ for some $f_0\in\mH$.
Quadratic estimates for the operator $T_A$ now shows that
\begin{multline*}
  \iint_{\R^{1+n}_+} |\nabla_{t,x}U|^2\, dtdx =\int_0^\infty \|F_t\|^2 dt \\
  = \int_0^\infty \| (t|T_A|)^{1/2} e^{-t|T_A|} f_0 \|^2\frac {dt}t\approx \| f_0 \|^2
  <\infty.
\end{multline*}
Thus $U\in \dot H^1(\R^{1+n}_+;\C^m)$.
To verify (\ref{eq:lmneu}), let $V\in \dot H^1(\R^{1+n}_+;\C^m)$ and consider the function 
$$
  g(t) := \int_{\R^n} ((A F_t)_0, V_t)\, dx, \qquad t>0,
$$
where we view $t\mapsto (A F_t)_0$ as a $C^\infty$ curve in $\ran(\divv_x; L_2)$ and 
$t\mapsto V_t$ as a continuous curve in $\dot H^{1/2}(\R^n;\C^m)$.
If $V_t\in C^1(\R_+;\dot H^{1/2})$, then
$$
  g'(t)=  \int_{\R^n} \big( (-\divv_x AF_t, V_t)+ (AF_t, \pd_t V_t) \big)\, dx = \int_{\R^n} (AF_t, \nabla_{t,x} V)\, dx.
$$
Hence $g(T)-g(\epsilon)= \iint_{\epsilon <t<T} (AF_t, \nabla_{t,x} V)\, dtdx$.
This also holds for general $V\in \dot H^1(\R^{1+n}_+;\C^m)$, which can be shown by mollifying $t\mapsto V_t$.
Taking limits $(\epsilon, T)\rightarrow (0,\infty)$ proves (\ref{eq:lmneu}).
\end{proof}

\begin{proof}[Proof for (Reg-$A$)]
Similar to the proof for the Neumann problem, we consider the equation
$$
  \nabla_x u= N^+\tilde f = 2N^+\chi_+(D\hat A)h = (I+ N^+ \sgn(D\hat A))h
$$
for $h\in N^+\mH$, in the tangential subspace.
We deduce that the trace $f$ of the solution is in the range of $T_A$, and therefore
$U\in \dot H^1(\R^{1+n}_+;\C^m)$ and (\ref{eq:lmdir}) follows as in the proof for the Neumann problem.

To prove that $U- P_t u\in \dot H^1_0(\R^{1+n}_+;\C^m)$, it suffices to show that $\pd_0 H_j= \pd_j H_0$ 
on $\R^{1+n}$, for $j=1,\ldots, n$, when $H$ is $\nabla_{t,x}( U- P_t u )$ extended by zero.
To this end, let $\Phi\in C^\infty_0(\R^{1+n})$ and consider the function
$$
  g_j(t):= \int_{\R^n} H_j(t,x) \Phi(t,x)\, dx, \quad t>0.
$$
Since $H$ is curl-free on $\R^{1+n}_+$, we have 
$\pd_0g_j=\int_{\R^n} H_0(-\pd_j \Phi)+ H_j\pd_0 \Phi$.
As we have $L_2(\R^n)$ convergence
$$
  H_j= \pd_j U_t - \pd_j P_t u \longrightarrow \pd_j u - \pd_j u =0,\qquad
  t\longrightarrow 0,
$$
integration of $\pd_0 g_j$ over $t\in (0, \infty)$ shows that
$\iint H_j\pd_0 \Phi= \iint H_0 \pd_j\Phi$ for all $\Phi\in C^\infty_0(\R^{1+n})$,
i.e. $\pd_0 H_j= \pd_j H_0$.
\end{proof}

\begin{proof}[Proof for (Dir-$A$)]
   Let $F=\nabla_{t,x} U$ be the boundary equation solution to (Neu$^\perp$-$A$) with data
   $\varphi=-u \in H^1(\R^n;\C^m)$, and recall that $U= F_0$ is the solution to (Dir-$A$).
   Consider the boundary trace $f=F|_{\R^n}\in E^+_A\mH$.
   We now instead use the isomorphism $\uA$ from (\ref{eq:similarity}), and define the 
   similar Hardy function $\tilde f:= \uA f\in \ran(\chi_+(\hat A D))$.

   The boundary condition on $\tilde f$ can be written
   $N^-\tilde f= N^- f= u$.
  We solve for $\tilde f$ using the ansatz $\tilde f= 2\chi_+(\hat A D)h$, where
  $h\in N^-\mH$. This yields the equation
$$
  u= N^-\tilde f = 2N^-\chi_+(\hat A D)h = (I+ N^- \sgn(\hat A D))h
$$
for $h$.
We note that the double layer type operator 
$K_A:= N^- \sgn(\hat A D)N^-$ maps 
$K_A(\dom(D)) \subset \dom(D)$
since 
$$ 
  DK_Ag= N^+ D\sgn(\hat A D)N^-g= N^+\sgn(D\hat A) DN^-g=  N^+\sgn(D \hat A )N^+(Dg)
$$ 
when $g\in\dom(D)$.
As above, since $\|K_A\|$ is small when $\|A- A_0\|_\infty$ is small, we can expand $(I+K_A)^{-1}$ 
in Neumann
series and deduce that $h\in\dom(D)$ since $u\in\dom(D)$.
This shows that $\tilde f= 2\chi_+(\hat A D)h\in \dom(\hat A D)$, and thus $f\in \dom(T_A)$.
In particular $f\in\dom(|T_A|^{1/2})$.
Quadratic estimates for the operator $T_A$ now shows that
\begin{multline*}
  \iint_{\R^{1+n}_+} |\nabla_{t,x}U|^2\, dtdx =\int_0^\infty \|\pd_t F_t\|^2 dt =\int_0^\infty \| T_A F_t\|^2 dt \\
  = \int_0^\infty \| (t|T_A|)^{1/2} e^{-t|T_A|}( |T_A|^{1/2}f) \|^2\frac {dt}t\approx \| |T_A|^{1/2}f \|^2
  <\infty.
\end{multline*}
Thus $U\in \dot H^1(\R^{1+n}_+;\C^m)$, and (\ref{eq:lmneu}) follows as in the proof for the Neumann problem.
Finally, note that $\nabla_{t,x} U= -T_A F_t= -e^{-t|T_A|}(T_A f)$. This shows $L_2$ convergence
$$
  \nabla_x U_t \longrightarrow -(T_A f)_\ta  =\nabla_x f_0= \nabla_x u.
$$
As in the proof for the regularity problem, $U- P_t u\in \dot H^1_0(\R^{1+n}_+;\C^m)$ follows.
\end{proof}
\begin{rem}
(i)
  Note that for any $L_2$ boundary data, the solution to the Neumann and regularity problem
  always satisfies $\int_0^1 \|\nabla_{t,x}U\|^2 \, dt <\infty$, whereas the solution to the Dirichlet 
  problem always satisfies $\int_1^\infty \|\nabla_{t,x}U\|^2 \, dt <\infty$. 
  Thus the problem whether good boundary data give $\dot H^1$ solutions concerns large $t$
   for the Neumann and regularity problem, and small $t$ for the Dirichlet problem. 

(ii)
  The structure of the problem in Lemma~\ref{lem:unique} is best explained abstractly as follows.
  Let $\mH_1 \hookrightarrow\mH_0$ be a continuous and dense
  inclusion of Hilbert spaces. Assume that $T_0:\mH_0\rightarrow \mH_0$ is an isomorphism which restricts
  to a bounded operator $T_1:\mH_1\rightarrow \mH_1$.
  It follows from \cite[Theorem 11.1]{Mcirr} that we have regularity $\mH_1= T_0^{-1}(\mH_1)$
  if and only if $T_1$ is a Fredholm operator and has index zero.
  See also \cite[Proposition 3.2.16]{Axthesis}.
  
  In our situation, $T_0= I+K$, $\mH_0= N^\pm\mH$ and $\mH_1$ is either $\ran(D)\cap N^\pm\mH$
  or $\dom(D)\cap N^-\mH$.
  In principle, the technique of Lemma~\ref{lem:unique} could be used to prove regularity
  for more general $A$ in the component of WP containing $I$. 
  The problem though is that in general the well posedness of two different BVP's, for the matrix $A$, 
  is needed both for the proof that $T_0$ is an isomorphism and that $T_1$ is Fredholm. 
  Index zero for $T_1$ could then be proved by the method of continuity, perturbing $A$ to $I$.
\end{rem}
\begin{proof}[Proof of Theorem~\ref{thm:uniqueness}]
Fix good boundary data and let $0<a<b<\infty$.
For all $A$ with the assumed properties, let $F=F_A=\nabla_{t,x}U$ be the solutions given
by Theorem~\ref{thm:bvpfordivform}, and let $F^{0}= F^0_A=\nabla_{t,x}U$ denote the
standard $\dot H^1$ solutions constructed with the Lax--Milgram Theorem as in 
Definition~\ref{defn:wellposed}.

From the quadratic estimates for $T_A$ it follows with arguments as in \cite[proof of Theorem 1.1]{AAH} 
that $A\mapsto F_A\in L_2(\R^n\times (a,b);\C^{(1+n)m})$ is analytic on WP.
The main result this uses is the analyticity of $A\mapsto b(T_A)$ for operators $b(T_A)$ in
the functional calculus of $T_A$. This was proved  in \cite[Theorem 6.4]{AKMc}.
Moreover, it is straightforward to verify that 
$A\mapsto F^0_A\in L_2(\R^{1+n}_+;\C^{(1+n)m})$ is analytic.
This means that whenever $A_0\in WP$, $\C\supset D\in z\mapsto A(z)$ are coefficients
depending analytically on a complex variable $z$, and
$A(0)=A_0$ and $H\in L_2(\R^n\times (a,b);\C^{(1+n)m})$, then the scalar function
$z\mapsto h(A(z))$, where
$$
  h(A):= \iint_{(a,b)\times \R^n} (F_{A}- F^0_{A},H)\, dtdx,
$$
is analytic on $D$.

Consider one of the BVP's and fix $A$ in the connected component of WP containing $I$.
Pick a sequence of balls $B_k= B(A_k;r_k)\subset WP$, $k=0, 1,\ldots, N$, such that $A_0=I$,
$A_N= A$ and $B_{k-1}\cap B_k\not = \emptyset$.
We may take $r_0<\epsilon$, so that $h=0$ on $B_0$ by Lemma~\ref{lem:unique}.
Now assume that $h=0$ on $B_{k-1}$ and pick any $A^1\in B_k$.
Let $A^0\in B_{k-1}\cap B_k$ and let 
$A(z):= (1-z) A^0 + z A^1$.
Then $h(A(z))$ vanishes on a neighbourhood of $0$. 
By analytic continuation $h(A^1)= h(A(1))=0$, and since $A^1\in B_k$ was
arbitrary, $h=0$ on $B_k$. We arrive at the conclusion that $h(A)=0$.
Since $a,b$ and $H$ are arbitrary, it follows
that $F_A=F_A^0$.
\end{proof}

\section{Boundary value problems for differential forms}   \label{sec:forms}

In this section, we demonstrate how Theorem~\ref{thm:bvpfordivform} and 
Theorem~\ref{thm:mainhardy} 
generalize to exterior/interior differential systems for 
$k$-vector fields, i.e.~differential forms of order $k$.

We use the notation from \cite[Section 2.1]{AAH}.
In particular, for fixed $k\in \{1,2,\ldots,n\}$, we consider functions
$$
   F(t,x) = \sum F_{\{s_1,\ldots, s_k\}}(t,x) \,\,e_{s_1}\wedg \ldots\wedg e_{s_k},
$$
taking values in the space $\wedge^k=\wedge^k\R^{1+n}$ of complex $k$-vectors on $\R^{1+n}$.
The vector fields in (\ref{eq:Laplacein1order}) is the special case $k=1$.
We point out that we assume the component functions $F_s$ to be scalar valued
here (m=1), although the methods apply, mutatis mutandis, to systems of exterior differential systems.
A natural generalization of the first order system (\ref{eq:Laplacein1order}) is the interior/exterior
differential system
\begin{equation} \label{eq:diracwedgek}
\begin{cases}   
  d^*_{t,x} B F(t,x)  =0, \\
  d_{t,x} F(t,x) =0,
\end{cases}
\end{equation}
where $F: \R^{1+n}_+\rightarrow \wedge^k$.
Here the exterior and interior derivative operators are
\begin{align*}
   d_{t,x}F  =\nabla_{t,x} \wedg F&= \sum_{j=0}^{n} e_j\wedg \pd_j F= \mu \pd_t F+ d_xF, \\
   d_{t,x}^*F  =-\nabla_{t,x} \lctr F&=- \sum_{j=0}^{n} e_j\lctr \pd_j F= -\mu^* \pd_t F+ d_x^*F, 
\end{align*}
where $\wedg$ denotes exterior product and $\lctr$ denotes (left) interior product, and
$\mu f= e_0\wedg f$ and $\mu^* f= e_0\lctr f$.
The matrix function $B\in L_\infty(\R^n;\mL(\wedge^k))$ is assumed to be $t$-independent and
{\em pointwise strictly accretive} in the sense that
$$  
  \re(B(x)w,w)\ge \kappa |w|^2,\qquad \text{for all } w\in \wedge^k \text{ and a.e. } x\in\R^n.
$$

To prove an analogue of Theorem~\ref{thm:mainhardy} for the Equation
(\ref{eq:diracwedgek}), we proceed as in Section~\ref{sec:ahat}
and introduce auxiliary matrices 
$$
  \oB:=
  \begin{bmatrix}
     B_{\no\no} & B_{\no\ta} \\ 0 & I
  \end{bmatrix}, \quad
  \uB:=
  \begin{bmatrix}
     I & 0 \\ B_{\ta \no} & B_{\ta\ta} 
  \end{bmatrix},
  \qquad\text{if }
  B=
  \begin{bmatrix}
     B_{\no\no} & B_{\no\ta} \\ B_{\ta \no} & B_{\ta\ta} 
  \end{bmatrix}
$$
in the normal/tangential splitting of $\wedge^k$.
Recall that a basis $k$-vector $e_{s_1}\wedg\ldots\wedg e_{s_k}$ 
is normal if one of the factors is $e_0$, and tangential otherwise.
Denote the tangential and normal parts of $f$ by $f_\ta$ and $f_\no$.
Splitting each of the Equations (\ref{eq:diracwedgek}) into normal and 
tangential parts and using the analogue of (\ref{eq:uoa}), shows that (\ref{eq:diracwedgek})
is equivalent to
$$
  \begin{cases}
    \pd_t (\oB F)_\no - \mu d_x^*(\uB F)_\ta = 0, \\
    \pd_t (\oB F)_\ta + \mu^* d_x(\uB F)_\no=0,
  \end{cases}
$$
together with the constraints $d_x F_\ta=0=d_x^*(BF)_\no$.
These tangential derivatives in the equations define the appropriate function
space
$$
  \mH^k_B:= \sett{f\in L_2(\R^n;\wedge^k)}{d_x f_\ta=0=d_x^*(Bf)_\no}
$$
generalizing $\mH$ from Definition~\ref{defn:Hspace}.
Note that when $k\ge 2$, the space $\mH^k_B$ depends on $B$, unlike the case $k=1$.

The normal derivatives in the equation give an equation $\pd_t F+ T_B F=0$, where
the infinitesimal generator is 
\begin{equation}    \label{eq:thatb}
    T_B:= \oB^{-1}D\uB.
\end{equation}
Here $D:=\mu^* d_x -\mu d_x^*$ is a self-adjoint differential operator.
The operator $T_B$ has similarities
\begin{equation*} 
  \oB^{-1}(D\hat B)\oB =T_B=  \uB^{-1}(\hat B D)\uB,
\end{equation*}
where $\hat B:= \uB \oB^{-1}$ is shown to be pointwise strictly accretive as in 
Proposition~\ref{prop:ahataccr}.
Thus Proposition~\ref{prop:typeomega} applies and proves that 
$T_B= \oB^{-1} (D \hat B)\oB$ is an $\omega$-bisectorial operator, $\omega$
being the angle of accretivity of $\hat B$.
Moreover, $T_B$ restricts to an injective $\omega$-bisectorial operator in 
$\mH^k_B=\oB^{-1}\clos{\ran(D)}$ with dense range, and in the splitting
$$
  L_2(\R^n;\wedge^k) = \clos{\ran(T_B)}\oplus \nul(T_B)= \mH^k_B \oplus
  \sett{[f_\no, f_\ta]^t}{d_x f_\no=0= d_x^*(Bf)_\ta},
$$
we have $T_B= T_B|_{\mH^k_B}\oplus 0$.

Similar to the proof of Theorem~\ref{thm:mainhardy}, Theorem~\ref{thm:qestsfordb} proves the 
boundedness of the Cauchy operator
$$
   E_B := \sgn(T_B)
$$
and the Hardy projections $E^\pm_B:= \chi_\pm(T_B)$.
To handle perturbation theory for the variable space $\mH_B^k$, we extend
this operator to $L_2(\R^n;\wedge^k)$ by defining $E_B^\pm f= E_B f=0$ when
$f\in\nul(T_B)$.

We obtain the
following result on Hardy space splittings of  $\mH^k_B\subset L_2(\R^n;\wedge^k)$.
\begin{thm}     \label{thm:hardyforforms}
Let $B \in L_\infty(\R^n;\mL(\wedge^k))$ be a $t$-independent, complex 
coefficient matrix function which is pointwise strictly accretive.

Then each $f\in \mH^k_B$ 
is in one-to-one correspondence with a pair of $k$-vector
fields $F^\pm_t=F^\pm(t,\cdot)\in C^1(\R_\pm; L_2(\R^n;\wedge^k))$ in $\R^{1+n}_\pm$
satisfying (\ref{eq:diracwedgek}) and having $L_2$ limits 
$\lim_{t\rightarrow 0^\pm}F^\pm_t=f^\pm$ and $\lim_{t\rightarrow \pm\infty}F^\pm_t=0$,
such that  
$$
f=f^++f^-.
$$
Under this correspondence, we have equivalences of norms $\|f\|_2\approx \|f^+\|_2 +\|f^-\|_2$ and
\begin{equation}   \label{eq:eqnormsforforms}
  \|f^\pm\|_2\approx \sup_{\pm t>0} \|F^\pm_t\|_2
  \approx \tb{t\pd_t F^\pm_{\pm t}}.
\end{equation}
Moreover, the Hardy space projections $L_2(\R^n;\wedge^k)\ni f\mapsto F^\pm=F^\pm_B:= e^{\mp t| T_B |}E_B^\pm f$
depend locally Lipschitz continuously on $B$ in the sense that
$$
  \|F^\pm_{B_2}-F^\pm_{B_1}\|_{\mX}\le C \|B_2-B_1\|_{L_\infty(\R^n)} \|f\|_2,
$$
where $C=C(\kappa_{B_1}, \kappa_{B_2}, \|B_1\|_\infty, \|B_2\|_\infty)$ and where
$\|F\|_\mX$ denotes any of the three norms in (\ref{eq:eqnormsforforms}).
\end{thm}
We remark that the proof of the non-tangential maximal estimate $\|\wt N_*(F)\|_2\approx \|f\|_2$
from \cite[Proposition 2.56]{AAH} in Theorem~\ref{thm:mainhardy} uses the divergence form 
structure of the second order system. This technique does not generalize to more 
general exterior differential systems.

Finally we extend the results in Section~\ref{sec:divbvp}  and show how 
Theorem~\ref{thm:hardyforforms} gives perturbation results for BVP's
for $k$-vector fields.
The natural BVP's are the following.
We are looking for a $k$-vector field $F_t\in C^1(\R_+;L_2(\R^n;\wedge^k))$ solving
(\ref{eq:diracwedgek}) in $\R^{1+n}_+$ with $L_2$ limits 
$\lim_{t\rightarrow 0^+}F_t=f$ and $\lim_{t\rightarrow \infty}F_t=0$,
where the boundary trace $f$ satisfies one of the following.
\begin{itemize}
\item Tangential boundary condition (Tan-$B$): $f_\ta = g$, 
where the given boundary data $g\in L_2(\R^n;\wedge^k)$ is tangential and satisfies
$d_x g=0$.
\item
Conormal boundary condition (Nor-$B$): $(Bf)_\no = g$,
where the given boundary data $g\in L_2(\R^n;\wedge^k)$ is normal and satisfies
$d_x^* g=0$.
\end{itemize}
Note that when $k=1$, (Tan-$B$) coincides with the Dirichlet regularity problem (Reg-$B$)
and (Nor-$B$) coincides with the Neumann problem (Neu-$B$).
\begin{thm}   \label{thm:bvpforkforms}
  The sets of well posedness $WP(Tan)$ and $WP(Nor)$
  are both open subsets of $L_\infty(\R^n; \mL(\wedge^k))$ and
  each contains
\begin{itemize}
\item[{\rm (i)}] all Hermitean matrices $B(x)= B(x)^*$ (and in particular all real symmetric matrices),
\item[{\rm (ii)}] all block matrices $B(x)= \begin{bmatrix} B_{\no\no}(x) & 0 \\ 0 & B_{\ta\ta}(x)\end{bmatrix}$, and
\item[{\rm (iii)}] all constant matrices $B(x)=B$.
\end{itemize}
\end{thm}
What is new here as compared with \cite{AAH}, is the perturbation result
around Hermitean and constant matrices, as well as the openness of the sets of well posedness.
The proof of Theorem~\ref{thm:bvpforkforms} is similar to the proof Theorem~\ref{thm:bvpfordivform}.
We observe that (Tan-$B$) and (Nor-$B$) are
well posed if and only if
\begin{align*}
  E^+_B\mH^k_B \longrightarrow \sett{g\in L_2(\R^n;\wedge^k)}{\mu^* g=0, \, d_x g=0}&:
  f\longmapsto f_\ta, \\
  E^+_B\mH^k_B \longrightarrow \sett{g\in L_2(\R^n;\wedge^k)}{\mu g=0, \, d^*_x g=0}&:
  f\longmapsto (Bf)_\no,
\end{align*}
are isomorphisms respectively.
Theorem~\ref{thm:hardyforforms} shows that $E_B$ depends continuously on $B$, so
we obtain from Lemma~\ref{lem:vardomains} that WP(Tan) and WP(Nor) are open sets.

That Hermitean and block matrices belong to WP is shown as in Sections~\ref{sec:hermitean} and
\ref{sec:block}, mutatis mutandis.
For constant matrices, the reverse Rellich argument used for second order divergence form elliptic systems $m>1$ 
applies.
For exterior differential systems, the symbol of the operator is
$$
  T_{\xi}:= \oB^{-1} D_\xi  \uB,
  \qquad 
  D_\xi:= 
  i(\mu^* \mu_\xi + \mu \mu_\xi^*),
$$
acting in the $2\binom{n-1}{k-1}$-dimensional space 
$\mH^k_\xi := \sett{f\in \wedge^k}{\mu\mu_\xi f= 0 = \mu^*\mu^*_\xi B f}$.
The proof uses that if $\mu\mu_\xi f= 0=\mu^*\mu_\xi^* f$, i.e. $f\in \ran(D_\xi)$, then
\begin{multline}   \label{eq:dxisquare}
  D_\xi^2 f= -( \mu^* \mu_\xi \mu \mu_\xi^*+\mu \mu_\xi^*   \mu^* \mu_\xi)f
  = ( \mu_\xi \mu^* \mu \mu_\xi^*+\mu _\xi^*\mu  \mu^* \mu_\xi)f \\
  = ( \mu_\xi (I-  \mu \mu^*) \mu_\xi^*+\mu _\xi^*(I- \mu^*  \mu) \mu_\xi)f
  =  ( \mu_\xi \mu_\xi^*+\mu _\xi^* \mu_\xi)f = |\xi|^2 f.
\end{multline}
For the anticommutation relations, we refer to \cite[Lemma 2.3]{AAH}.

\bibliographystyle{acm}

\end{document}